\documentclass[leqno,12pt]{amsart}
\usepackage{amssymb}
\usepackage{amsmath}
\usepackage{enumerate}
\usepackage{amsfonts}
\usepackage{hyperref}
\usepackage{mathrsfs}

\usepackage[headheight=18pt, top=20mm, bottom=20mm, left=22mm, right=22mm]{geometry}

\usepackage{tikz}
\usepackage{pgfplots}
\usetikzlibrary{arrows.meta}
\pgfplotsset{compat=1.18}

\definecolor{darkgreen}{RGB}{45, 119, 75}

\newcommand{\supp}{\text{\rm supp}}

\newtheorem{theorem}{Theorem}[section]
\newtheorem{corollary}[theorem]{Corollary}
\newtheorem{lemma}[theorem]{Lemma}
\newtheorem{proposition}[theorem]{Proposition}
\newtheorem{remark}[theorem]{Remark}
\newtheorem{definition}[theorem]{Definition}

\numberwithin{equation}{section}

\hypersetup{
    colorlinks=true,
    linkcolor= blue,
    citecolor =cyan,
    urlcolor = teal,
}

\makeatletter
\@ifundefined{c@part}{\newcounter{part}}{}

\renewcommand\part{%
  \@ifstar{\@spart}{\@part}%
}

\def\@part#1{%
  \refstepcounter{part}%
  \addcontentsline{toc}{part}{\protect\numberline{\thepart}#1}%
  \par\vspace{1\bigskipamount}
  \begin{center}%
    \normalfont\large\bfseries Part\ \thepart\quad #1%
  \end{center}%
  \vspace{1\bigskipamount}%
  \@afterheading
}

\def\@spart#1{%
  \addcontentsline{toc}{part}{#1}%
  \par\vspace{1\bigskipamount}
  \begin{center}%
    \normalfont\large\bfseries #1%
  \end{center}%
  \vspace{1\bigskipamount}%
  \@afterheading
}
\makeatother
  
\begin{document}

\title[Atomic Characterizations of Dunkl Hardy Spaces]{$H^p$ spaces in the Dunkl setting \\
meet Coifman--Weiss--type atoms}


\author[Jacek Dziubański]{Jacek Dziubański}
\author[Agnieszka Hejna-Łyżwa]{Agnieszka Hejna-Łyżwa}

\begin{abstract}
{Let $\Delta_k$ be the Dunkl Laplacian associated with an arbitrary
root system and a nonnegative multiplicity function. For every
$0<p\leq1$, a Coifman-Weiss type atomic  characterization of the 
Hardy space $H^p_{\mathrm{Dunkl}}$ is established. More precisely, it is proved that the space 
$H^p_{\mathrm{Dunkl}}$, which is originally defined by a relevant square function,  coincides with the space generated by
$({\rm CW},p,2)$-atoms, that is, atoms supported on Euclidean balls and satisfying cancellation
conditions against all polynomials of degree $ \leq s_p$, where
\[
    s_p=\left\lfloor \mathbf N\left(\frac1p-1\right)\right\rfloor 
\]
and $\mathbf N$ is the homogeneous dimension of the underlying
Dunkl measure. The corresponding  quasi-norms are   equivalent. 
We also show that the same space is obtained when the $L^2$ size
condition in the definition of  atoms is replaced by the 
$L^\infty$ size condition. The strategy  of the proof is to use an  operator-type atomic decomposition
associated with the Dunkl Laplacian, and then  prove that 
 each such operator atom  can be written    a linear combination of $({\rm CW},p,2)$-atoms.
}
\end{abstract}

\address{Jacek Dziubański, Uniwersytet Wroc\l awski,
Instytut Matematyczny,
Pl. Grunwaldzki 2,
50-384 Wroc\l aw,
Poland}
\email{jdziuban@math.uni.wroc.pl}

\address{Agnieszka Hejna-Łyżwa, Uniwersytet Wroc\l awski,
Instytut Matematyczny,
Pl. Grunwaldzki 2,
50-384 Wroc\l aw,
Poland}
\email{hejna@math.uni.wroc.pl}

\maketitle

\section{Introduction and main results}

Consider the Euclidean space $\mathbb R^N$ equipped with a root system $\mathcal R$ and a multiplicity function $k\geq 0$. In \cite{DzHL_atom}, for $0<p\leq 1$, the authors studied $H^p$ spaces in the Dunkl setting, which were originally defined by means of a square function associated with the Dunkl heat semigroup. To be more precise, let $H_t$ be the Dunkl heat semigroup generated by the Dunkl Laplace operator $\Delta_k$. Let $dw$ denote the associated measure, which is invariant under the reflection group $G$ associated with the root system $R$.

The classical real-variable theory of Hardy spaces, including their
atomic characterization, originates in the works of Fefferman and
Stein~\cite{Feff-Stein}, Coifman~\cite{Coifman}, and Coifman and
Weiss~\cite{CW}; the related theory of tent spaces was subsequently
introduced by Coifman, Meyer, and Stein~\cite{CMS}. Following the classical theory,  $H^p$-spaces in the Dunkl setting can be defined as follows.

\begin{definition}\label{def:Hp_Dunkl}
Let $0<p\leq1$.
For $f\in L^2(dw)$, define
\begin{equation}\label{eq:square_conic}
    Sf(\mathbf x)
    =
    \left(
        \int_0^\infty
        \int_{\|\mathbf x-\mathbf y\|<t}
        |t^2\Delta_k H_{t^2}f(\mathbf y)|^2
        \frac{dw(\mathbf y)}{w(B(\mathbf x,t))}
        \frac{dt}{t}
    \right)^{1/2}.
\end{equation}
Set
\[
    \mathbb H^p_{\mathrm{Dunkl}}
    =
    \{f\in L^2(dw): Sf\in L^p(dw)\},
    \qquad
    \|f\|_{\mathbb H^p_{\mathrm{Dunkl}}}
    =
    \|Sf\|_{L^p(dw)}.
\]
The Hardy space $H^p_{\mathrm{Dunkl}}$ is the completion
of $\mathbb H^p_{\mathrm{Dunkl}}$ with respect to this
quasi-norm.

More explicitly, two Cauchy sequences $(f_n)$ and $(g_n)$
in $\mathbb H^p_{\mathrm{Dunkl}}$ are identified if
\[
    \lim_{n\to\infty}
    \|f_n-g_n\|_{\mathbb H^p_{\mathrm{Dunkl}}}=0.
\]
The quasi-norm of the resulting equivalence class is
\[
    \|[(f_n)]\|_{H^p_{\mathrm{Dunkl}}}
    =
    \lim_{n\to\infty}
    \|f_n\|_{\mathbb H^p_{\mathrm{Dunkl}}}.
\]
\end{definition}

It was proved in \cite[Theorem 4.10]{DzHL_atom} that the space $H^p_{\rm Dunkl}$ admits atomic decomposition into atoms defined in the spirit of~\cite{Hofman} and~\cite{Duong} as follows. Let $\mathbf N$ denote the homogeneous dimension of the measure $dw$. Fix a positive integer $M>\mathbf{N}(2-p)/4p$. We say that a function $\boldsymbol a$ is a $ (p,2,M,\Delta_k)$-atom if there is a function $\boldsymbol b\in \mathcal D(\Delta_k^M)$ and a ball $B=B(\mathbf x_0,r_B)$  such that 
    \begin{equation*}
        \boldsymbol{a}=\Delta_k^M \boldsymbol{b},
    \end{equation*}
    \begin{equation*}
        \supp \, \boldsymbol{b}\subseteq \mathcal O(B)\quad (\text{\rm orbit of } B \ \text{\rm under the action of } G), 
    \end{equation*}
    and
\begin{equation*}
    \| (r_B^2 \Delta_k)^m \boldsymbol{b}\|_{L^2(dw)}\leq r_B^{2M} w(B)^{\frac{1}{2}-\frac{1}{p}}, \quad m=0,1,\dots, M.
\end{equation*}  
Theorem 4.10 of \cite{DzHL_atom} asserts that each element $f$ of $H^p_{\rm Dunkl}$ is identified with a tempered distribution  which can be written as:
$$ f=\sum_j \lambda_j \boldsymbol a_j, \quad \sum_j |\lambda_j|^p\leq C\| f\|_{H^p_{\rm Dunkl}}^p$$
where $\lambda_j\in\mathbb C$, $\boldsymbol a_j$ are $(p,2,M,\Delta_k)$-atoms. The convergence of the series is in $H^p_{\rm Dunkl}$ and in $\mathcal S'(\mathbb R^N)$ as well. Conversely, every series of $(p,2,M,\Delta_k)$-atoms with
$\ell^p$-summable coefficients converges in
$H^p_{\mathrm{Dunkl}}$ and in $\mathcal S'(\mathbb R^N)$, and its sum
satisfies
\[
    \big\|\sum_{j=1}^{\infty}\lambda_j\boldsymbol a_j
    \big\|_{H^p_{\mathrm{Dunkl}}}^p
    \leq C\sum_{j=1}^{\infty}|\lambda_j|^p.
\]

The purpose of this paper is to replace the operator atoms described
above by Coifman--Weiss-type atoms supported on single Euclidean balls
and satisfying classical polynomial cancellation conditions. This
passage is not immediate, since an operator atom is supported in the
full $G$-orbit of a ball, whose components may be widely separated.
The main difficulty is to localize such an atom without destroying its
moments and while retaining a uniform $\ell^p$ bound for the resulting
coefficients. We overcome this difficulty by a finite multiscale
decomposition of the orbit, combined with localized polynomial
corrections (see Theorem~\ref{prop:operator_to_polynomial_atoms}). The argument applies to every $0<p\leq1$, arbitrary root
systems, and nonnegative multiplicity functions.

For every nonnegative integer $s$, let
\[
    \mathcal P_s
    =
    \operatorname{span}_{\mathbb R}
    \{\mathbf x^\alpha:
      \alpha\in\mathbb N_0^N,\ |\alpha|\leq s\}
\]
denote the space of real polynomials of degree at most $s$.
For $0<p\leq1$, set
\[
    s_p
    =
    \left\lfloor
        \mathbf N\left(\frac1p-1\right)
    \right\rfloor.
\]

\begin{definition}[Coifman-Weiss type atoms, cf.~\cite{CW}]\label{def:p_q_atom}
Let $0<p\leq1$ and $q\in [1,\infty]$, $p<q$.
A measurable function $a$ on $\mathbb R^N$ is called
a $({\rm CW},p,q)$-atom if there exists a Euclidean ball
$B=B(\mathbf x_0,R)$, with $R>0$, such that:
\begin{enumerate}
    \item
    $a=0$ almost everywhere on
    $\mathbb R^N\setminus B$;
    \item
    \[
        \|a\|_{L^q(dw)}
        \leq w(B)^{1/q-1/p}.
    \]
    \item
    for every $P\in\mathcal P_{s_p}$,
    \[
        \int_{\mathbb R^N}
        a(\mathbf x)P(\mathbf x)\,dw(\mathbf x)=0.
    \]
\end{enumerate}
Each atom is identified with the tempered distribution
defined by
\[
    \langle a,\varphi\rangle
    =
    \int_{\mathbb R^N}
    a(\mathbf x)\varphi(\mathbf x)\,dw(\mathbf x),
    \qquad
    \varphi\in\mathcal S(\mathbb R^N).
\]
\end{definition}

\begin{definition}\label{def:Hp_atom}
Let $0<p\leq1$ and $q\in [1,\infty]$, $p<q$.
The atomic Hardy space $H^p_{\mathrm{CW},q}$ consists
of all tempered distributions
$F\in\mathcal S'(\mathbb R^N)$ admitting a representation
\[
    F=\sum_{j=1}^\infty\lambda_j a_j
    \quad\text{in }\mathcal S'(\mathbb R^N),
\]
where each $a_j$ is a $({\rm CW},p,q)$-atom and
$(\lambda_j)_{j\geq1}\in\ell^p(\mathbb C)$.
Here convergence in $\mathcal S'(\mathbb R^N)$ means that
\[
    \langle F,\varphi\rangle
    =
    \lim_{n\to\infty}
    \sum_{j=1}^n\lambda_j
    \int_{\mathbb R^N}
    a_j(\mathbf x)\varphi(\mathbf x)\,dw(\mathbf x)
\]
for every $\varphi\in\mathcal S(\mathbb R^N)$.

The quasi-norm is defined by
\[
    \|F\|_{H^p_{\mathrm{CW},q}}
    =
    \inf
    \Big\{
        \big(\sum_{j=1}^\infty|\lambda_j|^p\big)^{1/p}
        :
        F=\sum_{j=1}^\infty\lambda_j a_j
        \text{ in }\mathcal S'(\mathbb R^N),
        \quad
        a_j\text{ are }({\rm CW},p,q)\text{-atoms}
    \Big\}.
\]
We write $H^p_{\mathrm{CW}}=H^p_{\mathrm{CW},2}$.
\end{definition}

We are now in a position to state the main result of this paper.

\begin{theorem}\label{teo:H_p_coincides}
Let $0<p\leq1$.
Under the natural identification with tempered distributions,
\[
    H^p_{\mathrm{Dunkl}}
    =
    H^p_{\mathrm{CW},2}
    =
    H^p_{\mathrm{CW},\infty}
\]
with equivalent quasi-norms.
\end{theorem}

Thus, despite the nonlocal reflection structure of the Dunkl
Laplacian (see~\eqref{eq:laplace_formula}), the corresponding Hardy spaces admit a description in terms
of atoms having purely Euclidean supports and classical polynomial
cancellation.

\begin{remark}
    The theorem also yields an atomic decomposition in terms of
    Coifman--Weiss $({\rm CW},p,q)$-atoms for every $1\leq q\leq \infty$, $p<q$, because every
    $({\rm CW},p,\infty)$-atom is a $({\rm CW},p,q)$-atom, since
    \[
        \|a\|_{L^q(dw)}
        \leq w(B)^{1/q}\|a\|_{L^\infty(dw)}
        \leq w(B)^{1/q-1/p}.
    \]
    Thus, the decomposition obtained in the theorem is simultaneously
    a decomposition into $({\rm CW},p,q)$-atoms, with the same coefficients.

    Furthermore, if $s\geq s_p$, the atomic characterization may be
formulated using atoms satisfying cancellation conditions against all
polynomials in $\mathcal P_s$. The corresponding requirement on the
order of the operator atoms is explained in
Remark~\ref{rem:higher_order_cancellation}. 
\end{remark}

For $p=1$, the $H^1_{\rm Dunkl}$-spaces  were studied in \cite{ADzH} (see also~\cite{ABDH}) and were characterized there  by relevant square functions, Riesz transforms, maximal functions, and $(1,2,M,\Delta_k)$-atoms. In \cite{DH-atom} the authors proved characterization of $H^1_{\rm Dunkl}$ by Coifman-Weiss-type atoms with the cancellation condition against constant functions. The proof was based on estimates of kernels occurring in a Calder\'on reproducing formula.

The paper is organized as follows. After recalling the required facts
from Dunkl analysis (Section~\ref{sec:preliminaries}) and establishing auxiliary estimates for
polynomials (Section~\ref{sec:poly}), we decompose operator atoms into
Coifman--Weiss-type atoms (Section~\ref{subsec:operator_atoms}). We then prove the converse inclusion using
heat-kernel estimates (Section~\ref{sec:Hardy}) and conclude the equivalence of the corresponding
Hardy-space quasi-norms (Section~\ref{sec:bounded_atomic_decompositions}).

\section*{Acknowledgements}
The authors are grateful to Charles F. Dunkl for his careful reading
of the manuscript and for his kind and valuable comments.

\section{Dunkl analysis}\label{sec:preliminaries}
\subsection{Dunkl theory}
In this section, we present basic facts concerning the theory of the Dunkl operators.   For more details, we refer the reader to~\cite{Dunkl},~\cite{Roesle99},~\cite{Roesler3}, and~\cite{Roesler-Voit}. 

We consider the Euclidean space $\mathbb{R}^N$ with the scalar product $\langle \mathbf{x},\mathbf y\rangle=\sum_{j=1}^N x_jy_j
$, where $\mathbf x=(x_1,...,x_N)$, $\mathbf y=(y_1,...,y_N)$, and the norm $\| \mathbf x\|^2=\langle \mathbf x,\mathbf x\rangle$.

A {\it normalized root system}  in $\mathbb{R}^N$ is a finite set  $\mathcal R\subset \mathbb{R}^N\setminus\{0\}$ such that $\mathcal{R} \cap \alpha  \mathbb{R} = \{\pm \alpha\}$,  $\sigma_\alpha (\mathcal{R})=\mathcal{R}$, and $\|\alpha\|=\sqrt{2}$ for all $\alpha\in \mathcal{R}$, where $\sigma_\alpha$ is defined by $\sigma_\alpha (\mathbf x)=\mathbf x-2\frac{\langle \mathbf x,\alpha\rangle}{\|\alpha\|^2} \alpha$.
The finite group $G$ generated by the reflections $\sigma_{\alpha}$, $\alpha \in \mathcal{R}$, is called the {\it reflection group} of the root system. A~{\textit{multiplicity function}} is a $G$-invariant function $k:\mathcal{R}\to\mathbb C$, which will be fixed and non-negative throughout this paper.  

The associated $G$-invariant measure $dw$ is defined by $dw(\mathbf x)=w(\mathbf x)\, d\mathbf x$, where 
 \begin{equation}\label{eq:measure}
w(\mathbf x)=\prod_{\alpha\in \mathcal{R}}|\langle \mathbf x,\alpha\rangle|^{k(\alpha)}.
\end{equation}
Let $\mathbf{N}=N+\sum_{\alpha \in \mathcal{R}}k(\alpha)$. Then, 
\begin{equation}\label{eq:t_ball} w(B(t\mathbf x, tr))=t^{\mathbf{N}}w(B(\mathbf x,r)) \ \ \text{ for all } \mathbf x\in\mathbb{R}^N, \ t,r>0.
\end{equation}
Observe that there is a constant $C>1$ such that for all $\mathbf{x} \in \mathbb{R}^N$ and $r>0$, we have
\begin{equation}\label{eq:balls_asymp}
C^{-1}w(B(\mathbf x,r))\leq  r^{N}\prod_{\alpha \in \mathcal{R}} (|\langle \mathbf x,\alpha\rangle |+r)^{k(\alpha)}\leq C w(B(\mathbf x,r)),
\end{equation}
so $dw(\mathbf x)$ is doubling.

Moreover, by \eqref{eq:balls_asymp}, there exists a constant $C \geq 1$ such that for every $\mathbf{x} \in \mathbb{R}^N$,
\begin{equation}\label{eq:growth}
C^{-1} \left(\frac{R}{r}\right)^{ N}\leq \frac{w(B(\mathbf x, R))}{w(B(\mathbf x, r))}\leq C \left(\frac{R}{r}\right)^{\mathbf{N}}\ \ \text{ for } 0<r<R.
\end{equation}

For $\mathbf{x},\mathbf{y} \in \mathbb{R}^N$, let 
\begin{equation}\label{eq:distance_of_orbits}
    d(\mathbf x,\mathbf y)=\min_{\sigma\in G}\| \sigma(\mathbf x)-\mathbf y\|.    
\end{equation}

For a Lebesgue measurable set $A$ (in particular, for $A=B(\mathbf{x}_0,r)$) we denote
\begin{equation}
    \mathcal{O}(A)=\{\sigma(\mathbf{z})\;:\; \sigma \in G,\, \mathbf{z} \in A\}.
\end{equation}

For $\xi \in \mathbb{R}^N$, the {\it Dunkl operators} $T_\xi$  are the following $k$-deformations of the directional derivatives $\partial_\xi$ by   difference operators:
\begin{equation}\label{eq:T_xi}
     T_\xi f(\mathbf x)= \partial_\xi f(\mathbf x) + \sum_{\alpha\in \mathcal{R}} \frac{k(\alpha)}{2}\langle\alpha ,\xi\rangle\frac{f(\mathbf x)-f(\sigma_\alpha(\mathbf{x}))}{\langle \alpha,\mathbf x\rangle}.
\end{equation}
 We simply write $T_j$, if $\xi=e_j$, {where $\{e_j\}_{j=1}^N$ stands for the canonical basis in $\mathbb{R}^N$.}  

The Dunkl operators $T_{\xi}$, which were introduced in~\cite{Dunkl}, commute with each other and for suitable functions $f$ and $g$, the Dunkl operators satisfy
\[
    \int_{\mathbb R^N}T_jf(\mathbf x)g(\mathbf x)\,dw(\mathbf x)
    =
    -\int_{\mathbb R^N}f(\mathbf x)T_jg(\mathbf x)\,dw(\mathbf x).
\]
We shall use the following elementary degree property of Dunkl
operators:
\begin{equation}\label{eq:polynomial_lower}
     T_j\mathcal P_m\subseteq\mathcal P_{m-1},
    \qquad j=1,\ldots,N.
\end{equation}
where $\mathcal P_m=\{0\}$ for $m<0$.

For a multi-index $\boldsymbol{\beta}=(\beta_1,\beta_2,\ldots,\beta_N)\in \mathbb N_0^N$, we denote
\begin{equation}\label{eq:iterated_der_ord}
    |\boldsymbol{\beta}|=\beta_1+\ldots +\beta_N, \ \partial^{\mathbf{0}}=I, \ \ \partial^{\boldsymbol{\beta}}=\partial_1^{\beta_1} \circ \ldots \circ \partial_N^{\beta_N}, \ \ T^{\mathbf{0}}=I, \ \ T^{\boldsymbol{\beta}}=T_{1}^{\beta_1} \circ \ldots \circ T_{N}^{\beta_N}.
\end{equation}
In what follows, we denote by
\[
    \nabla
    =
    \left(\partial_{x_1},\ldots,\partial_{x_N}\right)
    \qquad\text{and}\qquad
    \nabla_k
    =
    \left(T_1,\ldots,T_N\right)
\]
the usual Euclidean gradient and the Dunkl gradient, respectively.

We shall also use the following comparison between Dunkl and
Euclidean derivatives, proved in~\cite[(6.4)]{DzHL_atom}: for every
multi-index $\beta$,
\begin{equation}\label{eq:small_x}
    \|T^\beta f\|_{L^\infty}
    \leq
    C_\beta
    \sum_{\beta'\in\mathbb N_0^N, |\beta'|=|\beta|
                    }
    \|\partial^{\beta'}f\|_{L^\infty}
\end{equation}
for every $f\in C^\infty(\mathbb R^N)$ for which the right-hand side
is finite.

\subsection{The Dunkl kernel and the Dunkl transform}

For a fixed $\mathbf y\in\mathbb R^N$, the Dunkl kernel
$\mathbf x\mapsto E(\mathbf x,\mathbf y)$ is the unique analytic
solution of
\begin{equation}\label{eq:Dunkl_kernel_definition}
    T_\xi f(\mathbf x)
    =
    \langle\xi,\mathbf y\rangle f(\mathbf x),
    \qquad
    f(\mathbf 0)=1.
\end{equation}
The function $E$ extends uniquely to a holomorphic function on
$\mathbb C^N\times\mathbb C^N$ and satisfies
\[
    E(\mathbf z,\mathbf w)=E(\mathbf w,\mathbf z),
    \qquad
    \mathbf z,\mathbf w\in\mathbb C^N.
\]
The positivity theorem for the Dunkl intertwining operator, proved
by R\"osler~\cite{Roesle99}, implies the sharp estimate
\begin{equation}\label{eq:Dunkl_kernel_bound}
    |E(i\boldsymbol\xi,\mathbf x)|
    \leq 1,
    \qquad
    \mathbf x,\boldsymbol\xi\in\mathbb R^N.
\end{equation}

For $f\in L^1(dw)$, its Dunkl transform is defined by
\begin{equation}\label{eq:Dunkl_transform}
    \mathcal Ff(\boldsymbol\xi)
    =
    \mathbf c_k^{-1}
    \int_{\mathbb R^N}
    f(\mathbf x)E(\mathbf x,-i\boldsymbol\xi)\,dw(\mathbf x),
    \qquad
    \mathbf c_k
    =
    \int_{\mathbb R^N}
    e^{-\|\mathbf x\|^2/2}\,dw(\mathbf x).
\end{equation}
In particular,~\eqref{eq:Dunkl_kernel_bound} gives
\begin{equation}\label{eq:Dunkl_transform_L1_Linfty}
    \|\mathcal Ff\|_{L^\infty(dw)}
    \leq
    \mathbf c_k^{-1}\|f\|_{L^1(dw)}.
\end{equation}

The Dunkl transform was introduced by C.F. Dunkl in
\cite{D5}, where the Plancherel theorem was also established.
Subsequently, de Jeu~\cite{dJ1} obtained a uniform bound for the
kernel $E(\mathbf x,i\boldsymbol\xi)$, which in particular yields
the $L^1(dw)$-to-$L^\infty(dw)$ estimate for the Dunkl transform.
R\"osler's positivity theorem~\cite{Roesle99} later yielded the
sharp estimate~\eqref{eq:Dunkl_kernel_bound}.

The Dunkl transform is an automorphism of
$\mathcal S(\mathbb R^N)$ and satisfies
\begin{equation}\label{eq:T_j_transform}
    \mathcal F(T_jf)(\boldsymbol\xi)
    =
    i\xi_j\mathcal Ff(\boldsymbol\xi),
    \qquad
    f\in\mathcal S(\mathbb R^N),
    \quad j=1,\ldots,N.
\end{equation}
Moreover, it extends uniquely to an isometry on $L^2(dw)$:
\begin{equation}\label{eq:Plancherel}
    \|\mathcal Ff\|_{L^2(dw)}
    =
    \|f\|_{L^2(dw)},
    \qquad f\in L^2(dw).
\end{equation}
For the original Plancherel theorem, see
\cite[Corollary~2.7]{D5}; see also
\cite[Theorem~4.26]{dJ1}.
\subsection{Dunkl Laplacian, Dunkl heat semigroup and Dunkl heat kernel}
The \textit{Dunkl Laplacian} associated with $\mathcal{R}$ and $k$  is the differential-difference operator $\Delta_k=\sum_{j=1}^N T_{j}^2$.
It was introduced in~\cite{Dunkl}, where it was also proved that $\Delta_k$ acts on $C^2(\mathbb{R}^N)$ functions by
\begin{equation}\label{eq:laplace_formula}
    \Delta_k f(\mathbf x)=\Delta f(\mathbf x)+\sum_{\alpha\in \mathcal{R}} k(\alpha) \delta_\alpha f(\mathbf x), \text{ where }\delta_\alpha f(\mathbf x)=\frac{\partial_\alpha f(\mathbf x)}{\langle \alpha , \mathbf x\rangle} -  \frac{f(\mathbf x)-f(\sigma_\alpha (\mathbf x))}{\langle \alpha, \mathbf x\rangle^2}.
\end{equation}
Here $\Delta=\sum_{j=1}^{N}\partial_j^2$. It follows from~\eqref{eq:T_j_transform} that for all $\xi \in \mathbb{R}^N$ and $f \in \mathcal{S}(\mathbb{R}^N)$, we have
\begin{equation}\label{eq:Laplacian_on_Fourier_side}
    \mathcal{F}(\Delta_{k}f)(\xi)=-\|\xi\|^2\mathcal{F}f(\xi).
\end{equation} 
The operator $(-\Delta_{k},\mathcal{S}(\mathbb{R}^N))$ in $L^2(dw)$ is densely defined and closable. It is essentially self-adjoint on $L^2(dw)$ (see, for instance, \cite[Theorem\;3.1]{AH}).

The closure of $(-\Delta_{k},\mathcal{S}(\mathbb{R}^N))$ generates a strongly continuous and positivity-preserving contraction semigroup on $L^2(dw)$.The semigroup has the form
  \begin{equation}\label{eq:heat}
  H_t f(\mathbf x)=\mathcal F^{-1}(e^{-t\|\xi\|^2}\mathcal Ff(\xi))(\mathbf x)=\int_{\mathbb R^N} h_t(\mathbf x,\mathbf y)f(\mathbf y)\, dw(\mathbf y),
  \end{equation}
where $ h_t(\mathbf x,\mathbf y)$
 is the so-called \textit{generalized heat kernel} (or the \textit{Dunkl heat kernel}), see~\cite{R1998} .  
  For all $t>0$ and $\mathbf{x},\mathbf{y} \in \mathbb{R}^N$,  one has $h_t(\mathbf{x},\mathbf{y})>0$, and $\int_{\mathbb{R}^N}h_t(\mathbf{x},\mathbf{z})\,dw(\mathbf{z})=1$.
 Formula~\eqref{eq:heat} defines  contraction semigroups on the $L^p(dw)$-spaces, $1\leq p\leq  \infty$, which are  strongly continuous for $1\leq p<\infty$. Note that in the case $k \equiv 0$ the Dunkl heat kernel is the classical heat kernel.

 \subsection{Heat-kernel estimates}
 
 In order to prove an atomic characterization of Hardy space, we need  bounds for the Dunkl heat kernel and its derivatives. 
 
 For $\mathbf{x},\mathbf{y} \in \mathbb{R}^N$ and $t,r>0$, we denote
\begin{equation}\label{eq:V}
    V(\mathbf{x},\mathbf{y},r):=\max\{w(B(\mathbf{x},r)),w(B(\mathbf{y},r))\}, \ \ \mathcal G_t(\mathbf x,\mathbf y)=\frac{1}{V(\mathbf x,\mathbf y,\sqrt{t})}e^{-\frac {d(\mathbf x,\mathbf y)^2}{t}}.
\end{equation} 

The following theorem was proved in~\cite{DH-atom}, see also~\cite[Theorem 4.1]{ADzH}.
For more detailed upper and lower bounds for $h_t(\mathbf x,\mathbf y)$ we refer the reader to~\cite{DH-heat}.

\begin{theorem}[Theorem 4.1, \texorpdfstring{~\cite{DH-atom}}{[DH]}]\label{teo:heat_new}   For every nonnegative integer $m$ and for all multi-indices $\boldsymbol{\alpha},\boldsymbol{\beta} \in \mathbb{N}_0^N$ there are constants $C_{m,\boldsymbol{\alpha},\boldsymbol{\beta}}, c>0$ such that
  \begin{equation}\label{eq:heat2} |\partial_t^m \partial_{\mathbf x}^{\boldsymbol{\alpha}}\partial_{\mathbf y}^{\boldsymbol{\beta}} h_t(\mathbf{x},\mathbf{y})|
  \leq C_{m,\boldsymbol{\alpha},\boldsymbol{\beta}} t^{-m-\frac{|\boldsymbol{\alpha}|}{2}-\frac{|\boldsymbol{\beta}|}{2}} \Big(1+\frac{\| \mathbf x-\mathbf y\|}{\sqrt{t}}\Big)^{-2} \mathcal G_{t\slash c} (\mathbf x,\mathbf y).
  \end{equation}
\end{theorem}

{
\subsection{Operator \texorpdfstring{$t^2\Delta_kH_{t^2}$}{Qt}}\label{sec:Qt}

For $t>0$, set $ Q_t=t^2\Delta_kH_{t^2}$. The integral kernel of $Q_t$ is
\begin{equation}\label{eq:Qt_kernel}
    Q_t(\mathbf x,\mathbf y)
    =
    t^2\Delta_{k,\mathbf x}h_{t^2}(\mathbf x,\mathbf y)
    =
    t^2\Delta_{k,\mathbf y}h_{t^2}(\mathbf x,\mathbf y).
\end{equation}
Thus, for $f\in L^2(dw)$,
\[
    Q_tf(\mathbf x)
    =
    \int_{\mathbb R^N}
    Q_t(\mathbf x,\mathbf y)f(\mathbf y)\,dw(\mathbf y).
\]
When no confusion is possible, we also write $Q_t*f=Q_tf$.

The heat-kernel estimates stated below imply that, for every
$\mathbf x\in\mathbb R^N$ and $t>0$, the function
$\mathbf y\mapsto Q_t(\mathbf x,\mathbf y)$ belongs to
$\mathcal S(\mathbb R^N)$. Hence, for $f\in\mathcal S'(\mathbb R^N)$,
we may define
\[
    Q_tf(\mathbf x)
    =
    \langle f,Q_t(\mathbf x,\cdot)\rangle.
\]
  
  Furthermore, from Theorem  \ref{teo:heat_new} we easily conclude that 

\begin{equation}\label{eq:Qt-bound} \begin{split}|\partial_t^m\partial_{\mathbf x}^{\boldsymbol\alpha}\partial_{\mathbf y}^{\boldsymbol\beta}Q_t(\mathbf x,\mathbf y)| & \leq  \frac{C_{m,\boldsymbol\alpha,\boldsymbol\beta}}{w(B(\mathbf x,t+d(\mathbf x,\mathbf y)))}t^{-m-|\boldsymbol\alpha|-|\boldsymbol\beta|}\exp\Big(-c\frac{d(\mathbf x,\mathbf y)^2}{t^2}\Big)\\
&\leq  \frac{C'_{m,\boldsymbol\alpha,\boldsymbol\beta}}{w(B(\mathbf x,t+d(\mathbf x,\mathbf y)))}t^{-m-|\boldsymbol\alpha|-|\boldsymbol\beta|}\exp\Big(-c\frac{d(\mathbf x,\mathbf y)}{t}\Big).\\
\end{split}\end{equation}
 } 

 We say that a tempered distribution $F$ coincides with a $dw(\mathbf x)$-locally integrable function $
f$ if 
$$ \langle F,\varphi\rangle= \int_{\mathbb{R}^N} f(\mathbf x)\varphi(\mathbf x)\, dw(\mathbf x) \quad \text{ for all } \varphi\in\mathcal S(\mathbb{R}^N).$$ 
Then we simply write $F=f$.

\section{Auxiliary estimates}\label{sec:poly}

\subsection{Estimates of polynomials with respect to the Dunkl measure}

\begin{lemma}\label{lem:polynomial_derivatives_sup}
There exists a constant $C_s>0$, depending only on $N$
and $s$, such that for every ball
$B=B(\mathbf x_0,r)$, with $r>0$, and every
$P\in\mathcal P_s$,
\[
    \sum_{m=0}^s\sum_{|\beta|=m}
    r^m\sup_{\mathbf x\in B}
    |\partial^\beta P(\mathbf x)|
    \leq
    C_s\sup_{\mathbf x\in B}|P(\mathbf x)|.
\]
\end{lemma}

\begin{proof}
On the finite-dimensional space $\mathcal P_s$, the expressions
\[
    \sup_{\mathbf x\in B(0,1)}|P(\mathbf x)|
    \quad\text{and}\quad
    \sum_{m=0}^s\sum_{|\beta|=m}
    \sup_{\mathbf x\in B(0,1)}
    |\partial^\beta P(\mathbf x)|
\]
define equivalent norms. This proves the assertion for
$B=B(0,1)$.

For a general ball $B=B(\mathbf x_0,r)$, apply this estimate
to the polynomial
\[
    \widetilde P(\mathbf u)=P(\mathbf x_0+r\mathbf u).
\]
Since
\[
    \partial^\beta\widetilde P(\mathbf u)
    =
    r^{|\beta|}
    (\partial^\beta P)(\mathbf x_0+r\mathbf u),
\]
the desired inequality follows.
\end{proof}

\begin{lemma}\label{lem:polynomial_sup_L2}
There exists a constant $C>0$, depending only on $N$, $s$,
and the doubling constant of $dw$, such that for every
ball $B=B(\mathbf x_0,r)$, with $r>0$, and every
$P\in\mathcal P_s$,
\[
    \sup_{\mathbf x\in B}|P(\mathbf x)|
    \leq
    \frac{C}{w(B)^{1/2}}
    \|P\|_{L^2(B,dw)}.
\]
\end{lemma}

\begin{proof}
The assertion is immediate for constant polynomials.
We may therefore assume that $s\geq1$ and $\nabla P\neq0$.
Set
\[
    V=\sup_{\mathbf x\in \overline{B}}|P(\mathbf x)|.
\]
There exists $\mathbf y_0\in \overline{B}$
such that $|P(\mathbf y_0)|=V$. Lemma~\ref{lem:polynomial_derivatives_sup} gives a constant
$C_s\geq1$ such that
\[
    \sup_{\mathbf x\in B}\|\nabla P(\mathbf x)\|
    \leq C_s r^{-1}V.
\]
For every $\mathbf x\in B$, the segment joining
$\mathbf y_0$ to $\mathbf x$ is contained in $ B$.
Consequently,
\begin{align*}
    |P(\mathbf x)|
    &\geq
    |P(\mathbf y_0)|-|P(\mathbf x)-P(\mathbf y_0)|\geq
    V-C_s r^{-1}V\|\mathbf x-\mathbf y_0\|.
\end{align*}
Set
\[
    \delta=\frac{1}{2C_s},
    \qquad
    E=B\cap B(\mathbf y_0,\delta r).
\]
Then
\[
    |P(\mathbf x)|\geq\frac{V}{2}
    \qquad\text{for every }\mathbf x\in E.
\]

By the doubling property, $w(E)\geq c\,w(B)$, with a constant
independent of $B$ and $P \in \mathcal{P}_s$. Hence,
\[
    \frac{V^2}{4}w(E)
    \leq
    \int_E|P(\mathbf x)|^2\,dw(\mathbf x)
    \leq
    \|P\|_{V^2(B,dw)}^2.
\]
Combining this with $w(E)\geq C_\delta^{-1}w(B)$
proves the assertion.
\end{proof}

\begin{corollary}\label{coro:polynomial_derivatives_L2}
There exists a constant $C>0$, depending only on $N$, $s$,
and the doubling constant of $dw$, such that for every
ball $B=B(\mathbf x_0,r)$, with $r>0$, and every
$P\in\mathcal P_s$,
\[
    \sum_{m=0}^s\sum_{|\beta|=m}
    r^m\sup_{\mathbf x\in B}
    |\partial^\beta P(\mathbf x)|
    \leq
    \frac{C}{w(B)^{1/2}}
    \|P\|_{L^2(B,dw)}.
\]
In particular, for every multi-index $\beta$ with
$|\beta|\leq s$,
\[
    \sup_{\mathbf x\in B}
    |\partial^\beta P(\mathbf x)|
    \leq
    \frac{C\,r^{-|\beta|}}{w(B)^{1/2}}
    \|P\|_{L^2(B,dw)}.
\]
\end{corollary}

\begin{proof}
Combine Lemmas~\ref{lem:polynomial_derivatives_sup}
and~\ref{lem:polynomial_sup_L2}.
\end{proof}

\subsection{Estimates of moments for truncated atoms}
Fix a function $\rho\in C_c^\infty(\mathbb R^N)$ such that
\[
    0\leq\rho\leq1,
    \qquad
    \rho=1\text{ on }B(0,1/4),
    \qquad
    \supp\rho\subseteq B(0,1/2).
\]
Fix positive integers $J,L$. 
Let $\mathcal Z$ be a nonempty finite set contained in
$B(\mathbf x_0,LR)$, where $R>0$ such that 
that $\#\mathcal Z\leq J$.
Define
\begin{equation}\label{eq:chi}
    \chi_{\mathbf x_0,R}(\mathbf x)
    =
    1-\prod_{\mathbf z\in\mathcal Z}
    \left(
        1-\rho\left(\frac{\mathbf x-\mathbf z}{R}\right)
    \right).
\end{equation}
Set
\[
    U=B(\mathbf x_0,(L+1)R).
\]
Then $\chi_{\mathbf x_0,R}\in C_c^\infty(U)$, $0\leq \chi_{\mathbf x_0,R}\leq 1$, $\chi_{\mathbf x_0,R} \equiv 0$ on $U^c$, and, for every
multi-index $\alpha$,
\begin{equation}\label{eq:cutoff_derivatives}
    \|\partial^\alpha\chi_{\mathbf x_0,R}\|_\infty
    \leq C_\alpha R^{-|\alpha|},
\end{equation}
where $C_\alpha$ depends only on $\alpha$, $J$, and $\rho$.
Moreover, $\chi_{\mathbf x_0,R}=1$ on
$\bigcup_{\mathbf z\in\mathcal Z}B(\mathbf z,R/4)$.

\begin{lemma}\label{lem:localized_polynomial_moments}
Let $M$ be a positive integer and let $s$ be a nonnegative
integer. Let $\chi_{\mathbf x_0,R}$ be the cutoff defined
above, and set
\[
    U=B(\mathbf x_0,(L+1)R).
\]
Suppose that $F\in\mathcal D(\Delta_k^M)$ vanishes
$dw$-almost everywhere outside $\mathcal O(B_*)$,
where $B_*=B(\mathbf y_0,r)$ and $r>0$.
Then, for every $P\in\mathcal P_s$,
\[
    \left|
        \int_{\mathbb R^N}
        \Delta_k^M F(\mathbf x)P(\mathbf x)
        \chi_{\mathbf x_0,R}(\mathbf x)\,dw(\mathbf x)
    \right|
    \leq
    C R^{-2M}
    \left(\frac{w(B_*)}{w(U)}\right)^{1/2}
    \|F\|_{L^2(dw)}\|P\|_{L^2(U,dw)}.
\]
The constant $C$ depends only on $M$, $s$, $L$, $J$, $\rho$,
and the fixed Dunkl structure.
\end{lemma}

\begin{proof}
Since $P\chi_{\mathbf x_0,R}\in C_c^\infty(\mathbb R^N)$
and $F\in\mathcal D(\Delta_k^M)$, self-adjointness gives
\[
    \int_{\mathbb R^N}
    \Delta_k^M F\,P\chi_{\mathbf x_0,R}\,dw
    =
    \int_{\mathbb R^N}
    F\,\Delta_k^M(P\chi_{\mathbf x_0,R})\,dw.
\]

The identity $\Delta_k=\sum_{j=1}^N T_j^2$ and the
commutativity of the Dunkl operators imply
$
    \Delta_k^M
    =
    \sum_{{\alpha\in\mathbb N_0^N, |\alpha|=M}}
    \frac{M!}{\alpha!}T^{2\alpha}.
$
Therefore, by~\eqref{eq:small_x},
\[
    \|\Delta_k^M(P\chi_{\mathbf x_0,R})\|_\infty
    \leq
    C\sum_{|\beta|=2M}
    \|\partial^\beta(P\chi_{\mathbf x_0,R})\|_\infty.
\]
Using Leibniz's rule and~\eqref{eq:cutoff_derivatives},
we obtain
\[
    \|\partial^\beta(P\chi_{\mathbf x_0,R})\|_\infty
    \leq
    C\sum_{{\gamma\leq\beta, |\gamma|\leq s}}
    R^{-(2M-|\gamma|)}
    \sup_{\mathbf x\in U}|\partial^\gamma P(\mathbf x)|,
    \qquad |\beta|=2M.
\]
By Corollary~\ref{coro:polynomial_derivatives_L2},
\[
    \sup_{\mathbf x\in U}|\partial^\gamma P(\mathbf x)|
    \leq
    \frac{C R^{-|\gamma|}}{w(U)^{1/2}}
    \|P\|_{L^2(U,dw)}.
\]
Consequently,
\[
    \|\Delta_k^M(P\chi_{\mathbf x_0,R})\|_\infty
    \leq
    \frac{C ((L+1)R)^{-2M}}{w(U)^{1/2}}
    \|P\|_{L^2(U,dw)}.
\]

Finally, the support assumption on $F$ and the
Cauchy--Schwarz inequality yield
\begin{align*}
    \left|
        \int_{\mathbb R^N}
        \Delta_k^M F\,P\chi_{\mathbf x_0,R}\,dw
    \right|
    &\leq
    \|\Delta_k^M(P\chi_{\mathbf x_0,R})\|_\infty
    \int_{\mathcal O(B_*)}|F|\,dw\\
    &\leq
    \frac{C ((L+1)R)^{-2M}}{w(U)^{1/2}}
    \|P\|_{L^2(U,dw)}
    w(\mathcal O(B_*))^{1/2}\|F\|_{L^2(dw)}\\
    &\leq
    C' R^{-2M}
    \left(\frac{w(B_*)}{w(U)}\right)^{1/2}
    \|F\|_{L^2(dw)}\|P\|_{L^2(U,dw)},
\end{align*}
where the last inequality follows from
$w(\mathcal O(B_*))\leq |G|w(B_*)$.
\end{proof}

\section{Operator atoms and localized polynomial moments}
\label{subsec:operator_atoms}

We recall the atomic characterization of
$H^p_{\mathrm{Dunkl}}$ established in~\cite{DzHL_atom}.
The atoms used in that characterization have cancellation
expressed through powers of the Dunkl Laplacian.
We shall use this structure to estimate their localized
polynomial moments.

For a positive integer $M$ we denote the domain of the operator $\Delta_k$ by

\begin{equation}\label{eq:domain}
    \mathcal{D}(\Delta_k^M)=\{f\in L^2(dw): \int_{\mathbb{R}^N} |\mathcal Ff(\xi)|^2\|\xi\|^{4M}\, dw(\xi)<\infty\}.
\end{equation}

\begin{definition}[{\cite[Definition~4.1]{DzHL_atom}}]
\label{def:operator_atom}
Let $0<p\leq1$ and let $M$ be a positive integer such that
\[
    M>\frac{\mathbf N(2-p)}{4p}.
\]
A function $\boldsymbol a$ is called a
$(p,2,M,\Delta_k)$-atom if there exist
$\boldsymbol b\in\mathcal D(\Delta_k^M)$ and a Euclidean
ball $B=B(\mathbf y_0,r)$, with $r>0$, such that
\[
    \boldsymbol a=\Delta_k^M\boldsymbol b,
\]
$\boldsymbol b$ vanishes $dw$-almost everywhere outside
$\mathcal O(B)$, and
\begin{equation}\label{eq:operator_atom_size}
    \|(r^2\Delta_k)^m\boldsymbol b\|_{L^2(dw)}
    \leq
    r^{2M}w(B)^{1/2-1/p},
    \qquad m=0,1,\ldots,M.
\end{equation}
We refer to these functions as operator atoms.
Each operator atom is identified with the tempered
distribution given by integration against $dw$.
\end{definition}

The following theorem summarizes the results of
\cite[Proposition~4.8, Theorem~4.10,
and Proposition~4.11]{DzHL_atom}.

\begin{theorem}\label{thm:operator_atomic_characterization}
Let $0<p\leq1$ and let $M$ be a positive integer satisfying
\[
    M>\frac{\mathbf N(2-p)}{4p}.
\]
Under the natural identification with tempered distributions,
a distribution $f\in\mathcal S'(\mathbb R^N)$ belongs to
$H^p_{\mathrm{Dunkl}}$ if and only if it admits
a representation
\[
    f=\sum_{j=1}^{\infty}\lambda_j\boldsymbol a_j
    \quad\text{in }\mathcal S'(\mathbb R^N),
\]
where the $\boldsymbol a_j$ are
$(p,2,M,\Delta_k)$-atoms and
$(\lambda_j)_{j\geq1}\in\ell^p(\mathbb C).
$
Moreover,
\[
    \|f\|_{H^p_{\mathrm{Dunkl}}}
    \asymp
    \inf
    \left(\sum_{j=1}^{\infty}|\lambda_j|^p\right)^{1/p},
\]
where the infimum is taken over all such representations.

Every series of operator atoms with coefficients in
$\ell^p$ converges unconditionally in
$\mathcal S'(\mathbb R^N)$ and converges in
$H^p_{\mathrm{Dunkl}}$.
If $f\in\mathbb H^p_{\mathrm{Dunkl}}$, the representation
can be chosen to converge also in $L^2(dw)$, with
\[
    \sum_{j=1}^{\infty}|\lambda_j|^p
    \leq C\|f\|_{\mathbb H^p_{\mathrm{Dunkl}}}^p.
\]
\end{theorem}

Recall that
\[
    s_p=\left\lfloor
        \mathbf N\left(\frac1p-1\right)
    \right\rfloor.
\]
The assumption on $M$ implies
\begin{equation}\label{eq:operator_atom_order}
    2M>
    \mathbf N\left(\frac1p-\frac12\right)
    =
    \mathbf N\left(\frac1p-1\right)+\frac{\mathbf N}{2}
    >s_p.
\end{equation}
Thus the operator cancellation is strong enough to imply
vanishing polynomial moments up to degree $s_p$, which is stated in the lemma below.

\begin{lemma}\label{lem:operator_atom_support_moments}
Let $\boldsymbol a$ be a $(p,2,M,\Delta_k)$-atom
associated with $B=B(\mathbf y_0,r)$.
Then $\boldsymbol a$ vanishes $dw$-almost everywhere
outside $\mathcal O(B)$,
\[
    \|\boldsymbol a\|_{L^2(dw)}
    \leq w(B)^{1/2-1/p},
\]
and
\[
    \int_{\mathbb R^N}
    \boldsymbol a(\mathbf x)P(\mathbf x)\,dw(\mathbf x)
    =0
    \qquad\text{for every }P\in\mathcal P_{s_p}.
\]
\end{lemma}

\begin{proof}
Write $\boldsymbol a=\Delta_k^M\boldsymbol b$
as in Definition~\ref{def:operator_atom}.
The size estimate follows from
\eqref{eq:operator_atom_size} with $m=M$.

To verify the support assertion, let
\[
    K=\bigcup_{\sigma\in G}
    {B(\sigma(\mathbf y_0),r)}.
\]
This set is $G$-invariant.
If $\varphi\in C_c^\infty(\mathbb R^N\setminus K)$,
then $\Delta_k^M\varphi$ vanishes on $K$, because
the Dunkl Laplacian involves ordinary derivatives
and reflections from $G$.
Self-adjointness therefore gives
\[
    \int_{\mathbb R^N}\boldsymbol a\varphi\,dw
    =
    \int_{\mathbb R^N}\boldsymbol b\Delta_k^M\varphi\,dw
    =0.
\]
Hence $\boldsymbol a$ vanishes almost everywhere outside
$K$. Since the boundaries of Euclidean balls have zero
$dw$-measure, the claimed support assertion follows.

Finally, choose a $G$-invariant function
$\zeta\in C_c^\infty(\mathbb R^N)$ equal to $1$
on a $G$-invariant neighborhood of $K$.
For $P\in\mathcal P_{s_p}$, self-adjointness gives
\[
    \int_{\mathbb R^N}\boldsymbol aP\,dw
    =
    \int_{\mathbb R^N}\boldsymbol a\,\zeta P\,dw
    =
    \int_{\mathbb R^N}
    \boldsymbol b\,\Delta_k^M(\zeta P)\,dw.
\]
On $K$, we have
$\Delta_k^M(\zeta P)=\Delta_k^MP=0$,
since $2M>s_p$ (cf.~\eqref{eq:polynomial_lower}).
This proves the cancellation property.
\end{proof}

We now apply Lemma~\ref{lem:localized_polynomial_moments}
to the function $\boldsymbol b$ defining an operator atom.
This yields the estimate needed to transfer polynomial
moments between groups of orbit points.

\begin{lemma}\label{lem:operator_atom_localized_moments}
Let $F$ be a $(p,2,M,\Delta_k)$-atom associated with
a ball $B_*=B(\mathbf y_0,r)$.
Let $\chi_{\mathbf x_0,R}$ and
\[
    U=B(\mathbf x_0,(L+1)R)
\]
be as in Lemma~\ref{lem:localized_polynomial_moments}.
Then, for every $P\in\mathcal P_{s_p}$,
\begin{equation}\label{eq:operator_atom_localized_moments}
    \left|
        \int_{\mathbb R^N}
        F(\mathbf x)P(\mathbf x)
        \chi_{\mathbf x_0,R}(\mathbf x)\,dw(\mathbf x)
    \right|
    \leq
    C\left(\frac rR\right)^{2M}
    \frac{w(B_*)^{1-1/p}}{w(U)^{1/2}}
    \|P\|_{L^2(U,dw)}.
\end{equation}
The constant $C$ depends only on $p$, $M$, $L$, $J$,
$\rho$, and the fixed Dunkl structure.
\end{lemma}

\begin{proof}
By Definition~\ref{def:operator_atom}, there exists
$\boldsymbol b\in\mathcal D(\Delta_k^M)$ such that
\[
    F=\Delta_k^M\boldsymbol b,
    \qquad
    \boldsymbol b=0
    \quad dw\text{-almost everywhere outside }\mathcal O(B_*),
\]
and
\[
    \|\boldsymbol b\|_{L^2(dw)}
    \leq r^{2M}w(B_*)^{1/2-1/p}.
\]
Applying Lemma~\ref{lem:localized_polynomial_moments}
to $\boldsymbol b$, we obtain
\begin{align*}
    \left|
        \int_{\mathbb R^N}
        FP\chi_{\mathbf x_0,R}\,dw
    \right|
    &=
    \left|
        \int_{\mathbb R^N}
        \Delta_k^M\boldsymbol b\,
        P\chi_{\mathbf x_0,R}\,dw
    \right|\leq
    C R^{-2M}
    \left(\frac{w(B_*)}{w(U)}\right)^{1/2}
    \|\boldsymbol b\|_{L^2(dw)}
    \|P\|_{L^2(U,dw)}\\
    &\leq
    C\left(\frac rR\right)^{2M}
    \frac{w(B_*)^{1-1/p}}{w(U)^{1/2}}
    \|P\|_{L^2(U,dw)}.
\end{align*}
This proves the assertion.
\end{proof}

In the subsequent decomposition, the cutoffs will be chosen
to isolate groups of orbit points at scale $R$.
Estimate~\eqref{eq:operator_atom_localized_moments}
provides the factor $(r/R)^{2M}$ needed to control
the corresponding polynomial corrections.

\begin{theorem}\label{prop:operator_to_polynomial_atoms}
Let $0<p\leq1$ and let $M$ be a positive integer satisfying
\[
    M>\frac{\mathbf N(2-p)}{4p}.
\]
There exists a constant $C>0$ such that every
$(p,2,M,\Delta_k)$-atom $\boldsymbol{a}$ admits a finite representation
\[
    {\boldsymbol{a}}=\sum_{j=1}^{J_{\boldsymbol{a}}}\lambda_j a_j,
\]
where each $a_j$ is a Coifman-Weiss-type $({\rm CW},p,2)$-atom and
\[
    \sum_{j=1}^{J_{\boldsymbol{a}}}|\lambda_j|^p\leq C.
\]
The constant $C$ is independent of ${\boldsymbol{a}}$ and its associated
ball. 
\end{theorem}

\begin{proof}
It suffices to consider real-valued functions.

Let ${\boldsymbol{a}}$ be associated with $B_*=B(\mathbf y_0,r)$.
Set
\[
    s=s_p=\left\lfloor \mathbf N\left(\frac1p-1\right)\right\rfloor,
    \qquad
    \gamma=\mathbf N\left(\frac1p-1\right),
    \qquad
    J=|G|.
\]
Our assumption on $M$ implies
\begin{equation}\label{eq:operator_decomp_exponents}
    2M>\gamma+\frac{\mathbf N}{2}>\gamma\geq s.
\end{equation}
By Lemma~\ref{lem:operator_atom_support_moments},
\begin{equation}\label{eq:operator_decomp_basic}
    \|{\boldsymbol{a}}\|_{L^2(dw)}\leq w(B_*)^{1/2-1/p},
    \qquad
    \int_{\mathbb R^N}{\boldsymbol{a}}P\,dw=0
    \quad(P\in\mathcal P_s),
\end{equation}
and ${\boldsymbol{a}}$ vanishes almost everywhere outside
$\mathcal O(B_*)$.

We first group the points of the orbit
\[
    Z=\mathcal O(\mathbf y_0)
\]
at successive scales
\[
    D_\ell=16\cdot2^\ell r,
    \qquad \ell=0,1,\ldots.
\]
For each $\ell$, consider the graph with vertex set $Z$
in which two distinct vertices are joined whenever their
Euclidean distance is less than $D_\ell$.
Let $\mathscr C_\ell$ denote the family of connected
components of this graph.

There are at most $J$ components at each scale.
Moreover, if $\mathbf{C}\in\mathscr C_\ell$, then
\begin{equation}\label{eq:component_geometry}
    \operatorname{diam}\mathbf{C}\leq(J-1)D_\ell.
\end{equation}
Points belonging to different components are at distance
at least $D_\ell$.
Since $D_{\ell+1}=2D_\ell$, each component in
$\mathscr C_\ell$ is contained in a unique component
in $\mathscr C_{\ell+1}$.

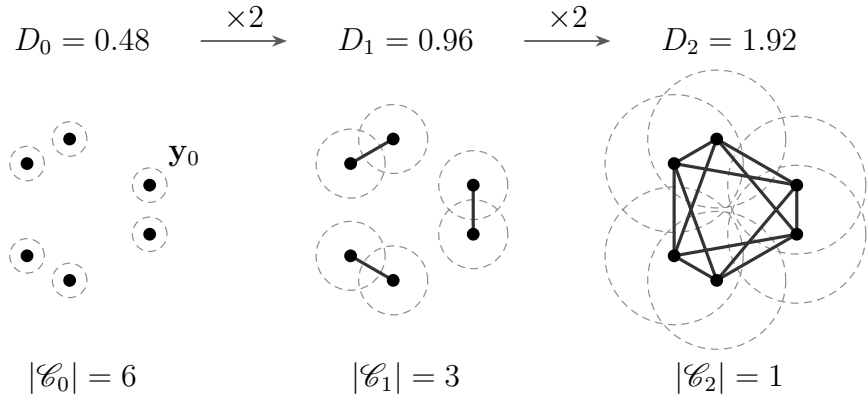
\begin{figure}[htbp]
\centering
\begin{tikzpicture}[
    scale=0.95,
    vertex/.style={
        circle,
        fill=black,
        inner sep=1.7pt
    },
    orbit ball/.style={
        densely dashed,
        draw=black!45,
        line width=0.45pt
    },
    graph edge/.style={
        draw=black!80,
        line width=1.2pt
    },
    scale arrow/.style={
        -{Stealth[length=2mm]},
        draw=black!65,
        line width=0.7pt
    }
]

%
%

\foreach \panel/\shift/\threshold/\level/\components in {
    0/0/0.48/0/6,
    1/4.5/0.96/1/3,
    2/9/1.92/2/1
}{
    \begin{scope}[xshift=\shift cm]

        \foreach \i/\angle in {
            1/20,
            2/340,
            3/140,
            4/100,
            5/260,
            6/220
        }{
            \coordinate (z\panel-\i) at (\angle:1);
        }

        \foreach \i in {1,...,6}{
            \draw[orbit ball]
                (z\panel-\i)
                circle[radius={0.5*\threshold}];
        }

        \foreach \i/\angleA in {
            1/20,
            2/340,
            3/140,
            4/100,
            5/260,
            6/220
        }{
            \foreach \j/\angleB in {
                1/20,
                2/340,
                3/140,
                4/100,
                5/260,
                6/220
            }{
                \ifnum\i<\j\relax
                    \pgfmathtruncatemacro{\joined}{
                        2-2*cos(\angleA-\angleB)
                        < \threshold*\threshold
                    }
                    \ifnum\joined=1\relax
                        \draw[graph edge]
                            (z\panel-\i)--(z\panel-\j);
                    \fi
                \fi
            }
        }

        \foreach \i in {1,...,6}{
            \node[vertex] at (z\panel-\i) {};
        }

        \node at (0,2.35)
            {$D_{\level}=\threshold$};

        \node at (0,-2.35)
            {$|\mathscr C_{\level}|=\components$};

    \end{scope}
}

\node[above right=3pt] at (z0-1) {$\mathbf y_0$};

\draw[scale arrow] (1.65,2.35)--(2.85,2.35)
    node[midway,above] {$\times2$};

\draw[scale arrow] (6.15,2.35)--(7.35,2.35)
    node[midway,above] {$\times2$};

\end{tikzpicture}

\caption{
The component graphs for a generic orbit of the
$A_2$ reflection group at three successive scales.
The vertices are the points of
$Z=\mathcal O(\mathbf y_0)$, and thick edges join
pairs at Euclidean distance less than $D_\ell$.
Dashed circles represent the auxiliary balls
$B(\mathbf z,D_\ell/2)$.
As the scale increases, components merge:
six singletons become three pairs and then one
connected component.
}
\label{fig:orbit_component_graphs}
\end{figure}

For $\mathbf{C}\in\mathscr C_\ell$, choose a point (vertex)
$\mathbf c_{\ell,\mathbf{C}}\in \mathbf{C}$ and set
\[
    E_{\mathbf{C}}=\bigcup_{\mathbf z\in {\mathbf{C}}}B(\mathbf z,r),
    \qquad
    {\boldsymbol{a}}_{\mathbf{C}}={\boldsymbol{a}}\chi_{E_{\mathbf{C}}},
    \qquad
    U_{\ell,{\mathbf{C}}}=B(\mathbf c_{\ell,{\mathbf{C}}},4JD_\ell).
\]
The sets $E_{\mathbf{C}}$, ${\mathbf{C}}\in\mathscr C_\ell$, are pairwise
disjoint. Thus, at every scale,
\begin{equation}\label{eq:operator_component_partition}
    {\boldsymbol{a}}=\sum_{{\mathbf{C}}\in\mathscr C_\ell}{\boldsymbol{a}}_{\mathbf{C}}.
\end{equation}
Each ${\boldsymbol{a}}_{\mathbf{C}}$ vanishes outside $U_{\ell,{\mathbf{C}}}$.
If ${\mathbf{C}}\in\mathscr C_{\ell+1}$, then
\begin{equation}\label{eq:operator_component_children}
    {\boldsymbol{a}}_{\mathbf{C}}
    =
    \sum_{{{\mathbf{C}}'\in\mathscr C_\ell,{\mathbf{C}}'\subseteq {\mathbf{C}}}}
    {\boldsymbol{a}}_{{\mathbf{C}}'}.
\end{equation}

For each component  ${\mathbf{C}}\in\mathscr C_\ell$, define
\[
    \chi_{\ell,{\mathbf{C}}}(\mathbf x)
    =
    1-\prod_{\mathbf z\in {\mathbf{C}}}
    \left(
        1-\rho\left(\frac{\mathbf x-\mathbf z}{D_\ell}\right)
    \right),
\]
where $\rho$ is the fixed cutoff used in
Lemma~\ref{lem:localized_polynomial_moments}.
Since $r\leq D_\ell/16$, this function equals $1$ on $E_{\mathbf{C}}$.
It equals $0$ on every $E_{{\mathbf{C}}'}$ with
${\mathbf{C}}'\in\mathscr C_\ell$ and ${\mathbf{C}}'\neq {\mathbf{C}}$, because distinct
components are separated by at least $D_\ell$.
Consequently,
\begin{equation}\label{eq:operator_cutoff_identity}
    {\boldsymbol{a}}\chi_{\ell,{\mathbf{C}}}={\boldsymbol{a}}_{\mathbf{C}}
    \quad dw\text{-almost everywhere}.
\end{equation}

By~\eqref{eq:component_geometry}, the cutoff
$\chi_{\ell,{\mathbf{C}}}$ satisfies the assumptions of
Lemma~\ref{lem:operator_atom_localized_moments}
with
\[
    \mathbf x_0=\mathbf c_{\ell,{\mathbf{C}}},
    \qquad R=D_\ell,
    \qquad L=4J-1,
    \qquad U=U_{\ell,{\mathbf{C}}}.
\]
Hence, for every $P\in\mathcal P_s$,
\begin{equation}\label{eq:operator_component_moments}
    \left|\int_{\mathbb R^N}{\boldsymbol{a}}_{\mathbf{C}}P\,dw\right|
    \leq
    C\left(\frac r{D_\ell}\right)^{2M}
    \frac{w(B_*)^{1-1/p}}{w(U_{\ell,{\mathbf{C}}})^{1/2}}
    \|P\|_{L^2(U_{\ell,{\mathbf{C}}},dw)}.
\end{equation}

We next represent these moments by localized polynomials.
Let $d_s=\dim\mathcal P_s$.
For every ball $U$, choose a real basis
$\{p_{U,n}\}_{n=1}^{d_s}$ of $\mathcal P_s$
that is orthonormal in $L^2(U,dw)$.
For a real-valued function $h\in L^2(dw)$ vanishing
outside $U$, set
\[
    R_Uh
    =
    \chi_U\sum_{n=1}^{d_s}
    \left(\int_Uh\,p_{U,n}\,dw\right)p_{U,n}.
\]
By orthonormality,
\begin{equation}\label{eq:operator_projection_properties}
\begin{split}
    \int_{\mathbb R^N}(R_Uh)P\,dw
    &=\int_{\mathbb R^N}hP\,dw
    \qquad(P\in\mathcal P_s),\\
    \|R_Uh\|_{L^2(dw)}&\leq\|h\|_{L^2(dw)},
    \qquad
    \|h-R_Uh\|_{L^2(dw)}\leq\|h\|_{L^2(dw)}.
\end{split}
\end{equation}
Define
\[
    q_{\ell,{\mathbf{C}}}=R_{U_{\ell,{\mathbf{C}}}}{\boldsymbol{a}}_{\mathbf{C}}.
\]
Applying~\eqref{eq:operator_component_moments}
to each $p_{U_{\ell,{\mathbf{C}}},n}$ gives
\begin{equation}\label{eq:operator_projection_decay}
    \|q_{\ell,{\mathbf{C}}}\|_{L^2(dw)}
    \leq
    C\sqrt{d_s}\left(\frac r{D_\ell}\right)^{2M}
    \frac{w(B_*)^{1-1/p}}{w(U_{\ell,{\mathbf{C}}})^{1/2}}.
\end{equation}

At the initial scale $\ell=0$, put
\[
    h_{\mathbf{C}}={\boldsymbol{a}}_{\mathbf{C}}-q_{0,{\mathbf{C}}},
    \qquad {\mathbf{C}}\in\mathscr C_0.
\]
Each $h_{\mathbf{C}}$, ${\mathbf{C}}\in\mathscr C_0$, vanishes outside $U_{0,{\mathbf{C}}}$ and has vanishing
moments against polynomials from $\mathcal P_s$.
By~\eqref{eq:operator_projection_properties}
and~\eqref{eq:operator_decomp_basic},
\[
    \|h_{\mathbf{C}}\|_{L^2(dw)}
    \leq\|{\boldsymbol{a}}_{\mathbf{C}}\|_{L^2(dw)}
    \leq w(B_*)^{1/2-1/p}.
\]
Since $\mathbf c_{0,{\mathbf{C}}}\in\mathcal O(\mathbf y_0)$
and the radius of $U_{0,{\mathbf{C}}}$ is $64Jr$,
$G$-invariance and doubling imply
\[
    w(U_{0,{\mathbf{C}}})\asymp w(B_*).
\]
Thus every nonzero $h_{\mathbf{C}}$ is a scalar multiple of a
$({\rm CW},p,2)$-atom, with a uniformly bounded coefficient.
There are at most $J$ such terms.

We now construct the terms joining successive scales.
For ${\mathbf{C}}\in\mathscr C_{\ell+1}$, define
\begin{equation}\label{eq:Hlc}
    H_{\ell,{\mathbf{C}}}
    =
    \sum_{{{\mathbf{C}}'\in\mathscr C_\ell, {\mathbf{C}}'\subseteq {\mathbf{C}}}}
    q_{\ell,{\mathbf{C}}'}
    -
    q_{\ell+1,{\mathbf{C}}}.
\end{equation}
Equations~\eqref{eq:operator_component_children}
and~\eqref{eq:operator_projection_properties} imply
\[
    \int_{\mathbb R^N}H_{\ell,{\mathbf{C}}}P\,dw=0
    \qquad(P\in\mathcal P_s).
\]
Furthermore, all the balls $U_{\ell,{\mathbf{C}}'}$ occurring
in the sum~\eqref{eq:Hlc} are contained in $U_{\ell+1,{\mathbf{C}}}$.
Indeed, both centers belong to ${\mathbf{C}}$, so
\[
    \|\mathbf c_{\ell,{\mathbf{C}}'}-\mathbf c_{\ell+1,{\mathbf{C}}}\|
    \leq(J-1)D_{\ell+1}.
\]
Since the radius of $U_{\ell,{\mathbf{C}}'}$ is
$4JD_\ell=2JD_{\ell+1}$, the claimed inclusion follows.
Hence $H_{\ell,{\mathbf{C}}}$ vanishes outside $U_{\ell+1,{\mathbf{C}}}$.

The same geometry and doubling give
\[
    w(U_{\ell,{\mathbf{C}}'})\asymp w(U_{\ell+1,{\mathbf{C}}})
    \qquad({\mathbf{C}}'\subseteq {\mathbf{C}}).
\]
There are at most $J$ terms.
Using~\eqref{eq:operator_projection_decay}, we obtain
\begin{equation}\label{eq:operator_transition_L2}
    \|H_{\ell,{\mathbf{C}}}\|_{L^2(dw)}
    \leq
    C\left(\frac r{D_\ell}\right)^{2M}
    \frac{w(B_*)^{1-1/p}}
         {w(U_{\ell+1,{\mathbf{C}}})^{1/2}}.
\end{equation}

For each nonzero $H_{\ell,{\mathbf{C}}}$, set
\[
    \lambda_{\ell,{\mathbf{C}}}
    =
    \|H_{\ell,{\mathbf{C}}}\|_2
    w(U_{\ell+1,{\mathbf{C}}})^{1/p-1/2},
    \qquad
    a_{\ell,{\mathbf{C}}}=\lambda_{\ell,{\mathbf{C}}}^{-1}H_{\ell,{\mathbf{C}}}.
\]
Then $a_{\ell,{\mathbf{C}}}$ is a $({\rm CW},p,2)$-atom associated with the ball $U_{\ell+1,\mathbf{C}}$.
By~\eqref{eq:operator_transition_L2},
\[
    |\lambda_{\ell,{\mathbf{C}}}|
    \leq
    C\left(\frac r{D_\ell}\right)^{2M}
    \left(
        \frac{w(U_{\ell+1,{\mathbf{C}}})}{w(B_*)}
    \right)^{1/p-1}.
\]
The center of $U_{\ell+1,{\mathbf{C}}}$ belongs to
$\mathcal O(\mathbf y_0)$.
Therefore, $G$-invariance and the growth estimate
\eqref{eq:growth} imply
\[
    \frac{w(U_{\ell+1,{\mathbf{C}}})}{w(B_*)}
    \leq C\left(\frac{D_\ell}{r}\right)^{\mathbf N}.
\]
Consequently,
\begin{equation}\label{eq:operator_transition_coefficients}
    |\lambda_{\ell,{\mathbf{C}}}|
    \leq
    C\left(\frac r{D_\ell}\right)^{2M-\gamma}
    \leq C2^{-\ell(2M-\gamma)}.
\end{equation}

It remains to verify the finite telescoping identity.
Set
\[
    g_\ell=\sum_{{\mathbf{C}}\in\mathscr C_\ell}q_{\ell,{\mathbf{C}}}.
\]
Choose an integer $\ell_*$ such that
$D_{\ell_*}>\operatorname{diam}Z$.
Then $\mathscr C_{\ell_*}=\{Z\}$ and ${\boldsymbol{a}}_Z={\boldsymbol{a}}$ (see \eqref{eq:operator_component_partition}).
By the global cancellation in
\eqref{eq:operator_decomp_basic},
\[
    g_{\ell_*}=R_{U_{\ell_*,Z}}{\boldsymbol{a}}=0.
\]
Thus
\begin{align*}
    {\boldsymbol{a}}
    &=
    ({\boldsymbol{a}}-g_0)
    +\sum_{\ell=0}^{\ell_*-1}(g_\ell-g_{\ell+1})=
    \sum_{{\mathbf{C}}\in\mathscr C_0}h_{\mathbf{C}}
    +
    \sum_{\ell=0}^{\ell_*-1}
    \sum_{{\mathbf{C}}\in\mathscr C_{\ell+1}}H_{\ell,{\mathbf{C}}}.
\end{align*}
When $\ell_*=0$, the second sum is empty.

Normalize the nonzero initial terms $h_{\mathbf{C}}$ in the same
way as the transition terms, and omit all zero terms.
Their coefficients have uniformly bounded $p$-sum.
For the remaining coefficients,
\eqref{eq:operator_transition_coefficients} and
\eqref{eq:operator_decomp_exponents} yield
\[
    \sum_{\ell=0}^{\ell_*-1}
    \sum_{{{\mathbf{C}}\in\mathscr C_{\ell+1},
                    H_{\ell,{\mathbf{C}}}\neq0}}
    |\lambda_{\ell,{\mathbf{C}}}|^p
    \leq
    C\sum_{\ell=0}^{\infty}
    2^{-\ell p(2M-\gamma)}
    \leq C.
\]
This proves the required uniform coefficient bound.

All sums in the construction are finite and all their
terms belong to $L^2(dw)$.
The resulting identity therefore holds in $L^2(dw)$
and, under the identification by integration against $dw$,
also in $\mathcal S'(\mathbb R^N)$.
\end{proof}

\begin{remark}\label{rem:higher_order_cancellation}
Let $s\in\mathbb N_0$ satisfy $s\geq s_p$, and suppose, in
addition, that $2M>s$. Then the proof of Theorem~\ref{prop:operator_to_polynomial_atoms}
applies verbatim with $\mathcal P_{s_p}$ replaced by
$\mathcal P_s$. Consequently, every
$(p,2,M,\Delta_k)$-atom admits a finite decomposition into
Coifman--Weiss-type $({\rm CW},p,2)$-atoms satisfying
\[
    \int_{\mathbb R^N}a_j(\mathbf x)P(\mathbf x)\,dw(\mathbf x)=0,
    \qquad P\in\mathcal P_s,
\]
with a uniform bound for the $\ell^p$-sum of the coefficients.

Indeed, the only additional requirement is that
$\Delta_k^M P=0$ for $P\in\mathcal P_s$, which follows from
\[
    \Delta_k^M\mathcal P_s\subseteq\mathcal P_{s-2M}=\{0\}.
\]
Thus, if $M$ is chosen sufficiently large, the atomic
characterization of $H^p_{\mathrm{Dunkl}}$ remains valid when
the atoms are required to have vanishing moments up to any fixed
degree $s\geq s_p$.
\end{remark}

\section{Proof of the Coifman--Weiss-type atomic characterization}\label{sec:Hardy}

We first record a basic fact concerning the atomic spaces.
It guarantees that atomic series define tempered distributions
and that convergence in the atomic quasi-norm implies
distributional convergence. This will allow us to identify
the distributions obtained from the operator atomic
decomposition with elements of the polynomial atomic space.

\begin{lemma}\label{lem:distributional_convergence}
Let $0<p\leq1$ and $1<q\leq\infty$.
Then:
\begin{enumerate}
    \item[(i)]
    For every sequence $(a_j)_{j\geq1}$ of $({\rm CW},p,q)$-atoms
    and every $(\lambda_j)_{j\geq1}\in\ell^p$,
    the series $\sum_{j=1}^\infty\lambda_j a_j$
    converges unconditionally in $\mathcal S'(\mathbb R^N)$.
    \item[(ii)]
    The inclusion
    $H^p_{\mathrm{CW},q} \subset\mathcal S'(\mathbb R^N)$
    is continuous.
\end{enumerate}
\end{lemma}

\begin{proof}
Set
\[
    s=s_p=\left\lfloor\mathbf N\left(\frac1p-1\right)\right\rfloor,
    \qquad
    \mathfrak q_p(\phi)
    =
    \|\phi\|_\infty
    +
    \sum_{|\beta|=s+1}\|\partial^\beta\phi\|_\infty.
\]
We first show that
\begin{equation}\label{eq:atom_test_bound}
    |\langle a,\phi\rangle|
    \leq C\mathfrak q_p(\phi)
\end{equation}
uniformly over all $({\rm CW},p,q)$-atoms $a$.

Let $a$ be associated with $B=B(\mathbf x_0,r)$.
By H\"older's inequality and the size condition,
with the convention $1/\infty=0$, we have
\[
    \|a\|_{L^1(dw)}
    \leq
    w(B)^{1-1/q}\|a\|_{L^q(dw)}
    \leq
    w(B)^{1-1/q}w(B)^{1/q-1/p}
    =
    w(B)^{1-1/p}.
\]
Recall from~\eqref{eq:balls_asymp} that
\[
    w(B(\mathbf x,r))\geq cr^{\mathbf N}.
\]
If $r\geq1$, it follows that
\[
    |\langle a,\phi\rangle|
    \leq w(B)^{1-1/p}\|\phi\|_\infty
    \leq C\|\phi\|_\infty.
\]

Suppose that $0<r<1$.
For a real-valued $\phi\in\mathcal S(\mathbb R^N)$,
let $P$ be its Taylor polynomial of degree $s$
at $\mathbf x_0$.
By cancellation and Taylor's theorem,
\begin{align*}
    |\langle a,\phi\rangle|
    &=
    \left|
        \int_B a(\mathbf x)
        \bigl(\phi(\mathbf x)-P(\mathbf x)\bigr)
        \,dw(\mathbf x)
    \right|\leq
    Cr^{s+1}w(B)^{1-1/p}
    \sum_{|\beta|=s+1}\|\partial^\beta\phi\|_\infty\\
    &\leq
    Cr^{s+1-\mathbf N(1/p-1)}
    \sum_{|\beta|=s+1}\|\partial^\beta\phi\|_\infty
    \leq C\mathfrak q_p(\phi),
\end{align*}
since $s+1>\mathbf N(1/p-1)$.
For complex-valued $\phi$, apply the same argument to
its real and imaginary parts.
This proves~\eqref{eq:atom_test_bound}.

Since $\ell^p\subseteq\ell^1$,
\[
    \sum_{j=1}^\infty
    |\lambda_j\langle a_j,\phi\rangle|
    \leq
    C\mathfrak q_p(\phi)\sum_{j=1}^\infty|\lambda_j|
    \leq
    C\mathfrak q_p(\phi)
    \left(\sum_{j=1}^\infty|\lambda_j|^p\right)^{1/p}.
\]
Thus
\[
    \langle F,\phi\rangle
    :=
    \sum_{j=1}^\infty\lambda_j\langle a_j,\phi\rangle
\]
defines a tempered distribution.
Moreover, for every finite set $I\subset\mathbb N$,
\[
    \left|
        \Big\langle
            F-\sum_{j\in I}\lambda_ja_j,\phi
        \Big\rangle
    \right|
    \leq
    C\mathfrak q_p(\phi)\sum_{j\notin I}|\lambda_j|.
\]
The right-hand side tends to zero as $I$ increases
to $\mathbb N$, proving~(i).

Finally, applying the preceding estimate to any atomic
representation of $f\in H^p_{\mathrm{CW},q}$ and
taking the infimum over all representations gives
\begin{equation}\label{eq:atomic_distribution_bound}
    |\langle f,\phi\rangle|
    \leq
    C\mathfrak q_p(\phi)\|f\|_{H^p_{\mathrm{CW},q}}.
\end{equation}
This proves~(ii), since $\mathfrak q_p$ is a continuous
seminorm on $\mathcal S(\mathbb R^N)$.
\end{proof}
We now prove the atomic characterization of
$H^p_{\mathrm{Dunkl}}$. The auxiliary results established above
allow us to treat the two inclusions separately. First, the operator
atomic decomposition and Theorem~\ref{prop:operator_to_polynomial_atoms}
yield
\[
    H^p_{\mathrm{Dunkl}}
    \subseteq
    H^p_{\mathrm{CW},2}.
\]
Conversely, the uniform square-function estimate for individual
$({\rm CW},p,2)$-atoms, together with completeness and distributional
convergence, gives
\[
    H^p_{\mathrm{CW},2}
    \subseteq
    H^p_{\mathrm{Dunkl}}.
\]

\begin{theorem}\label{thm:Hp_inclusion}
Let $0<p\leq1$.
Under the natural identification with tempered distributions,
\[
    H^p_{\mathrm{Dunkl}}\subseteq H^p_{\mathrm{CW},2},
    \qquad
    \|f\|_{H^p_{\mathrm{CW},2}}
    \leq C\|f\|_{H^p_{\mathrm{Dunkl}}}.
\]
\end{theorem}

\begin{proof}
Fix an integer $M>\mathbf N(2-p)/(4p)$.
By Theorem~\ref{thm:operator_atomic_characterization},
every $f\in H^p_{\mathrm{Dunkl}}$ admits a representation
\[
    f=\sum_{j=1}^{\infty}\mu_j\boldsymbol a_j
    \quad\text{in }\mathcal S'(\mathbb R^N),
    \qquad
    \sum_{j=1}^{\infty}|\mu_j|^p
    \leq C\|f\|_{H^p_{\mathrm{Dunkl}}}^p,
\]
where the $\boldsymbol a_j$ are $(p,2,M,\Delta_k)$-atoms.
By Theorem~\ref{prop:operator_to_polynomial_atoms},
\[
    \boldsymbol a_j=\sum_{k=1}^{J_j}\lambda_{j,k}a_{j,k},
    \qquad
    \sum_{k=1}^{J_j}|\lambda_{j,k}|^p\leq C,
\]
where the $a_{j,k}$ are $({\rm CW},p,2)$-atoms and the constant
is independent of $j$. Hence
\[
    \sum_{j=1}^{\infty}\sum_{k=1}^{J_j}
    |\mu_j\lambda_{j,k}|^p
    \leq C\|f\|_{H^p_{\mathrm{Dunkl}}}^p.
\]
By Lemma~\ref{lem:distributional_convergence}, the series
\[
    \sum_{j=1}^{\infty}\sum_{k=1}^{J_j}
    \mu_j\lambda_{j,k}a_{j,k}
\]
converges unconditionally in $\mathcal S'(\mathbb R^N)$.
Its partial sums over the first $n$ complete blocks equal
$\sum_{j=1}^n\mu_j\boldsymbol a_j$ and therefore converge
to $f$. Thus the series is an atomic representation of $f$,
and the definition of the atomic quasi-norm yields
\[
    \|f\|_{H^p_{\mathrm{CW},2}}^p
    \leq
    \sum_{j=1}^{\infty}\sum_{k=1}^{J_j}
    |\mu_j\lambda_{j,k}|^p
    \leq C\|f\|_{H^p_{\mathrm{Dunkl}}}^p.
\]
\end{proof}

\begin{proposition}\label{prop:atom_square_bound}
    Fix $0<p\leq1$. There is a constant $C>0$ such that
    for every Coifman--Weiss $({\rm CW},p,2)$-atom $a$, with
    cancellation against $\mathcal P_{s_p}$, one has
    \begin{equation}
        \|Sa\|_{L^p(dw)}\leq C.
    \end{equation}
\end{proposition}

\begin{proof}
    Suppose that $a$ is associated with a ball
    $B=B(\mathbf x_0,r)$.
    Since $S$ is bounded on $L^2(dw)$, using the
    doubling property and the $G$-invariance of $dw$,
    we get
    \begin{equation}
    \begin{split}
        \|Sa\|_{L^p(\mathcal O(4B),dw)}^p
        &=
        \int_{\mathcal O(4B)}
        |Sa(\mathbf x)|^p\,dw(\mathbf x)\leq
        \left(
            \int_{\mathcal O(4B)}
            |Sa(\mathbf x)|^2\,dw(\mathbf x)
        \right)^{p/2}
        w(\mathcal O(4B))^{1-p/2}\\
        &\leq
        C\|a\|_{L^2(dw)}^p w(B)^{1-p/2}
        \leq C.
    \end{split}
    \end{equation}

    We now estimate $Sa(\mathbf x)$ for
    $d(\mathbf x,\mathbf x_0)>4r$.
    Let $P(\mathbf z)$ be the Taylor polynomial of
    degree $s_p$ of the function
    $\mathbf z\mapsto Q_t(\mathbf y,\mathbf z)$
    at $\mathbf x_0$.
    By~\eqref{eq:Qt-bound}, symmetry of the kernel,
    and doubling, for $\|\mathbf x-\mathbf y\|<t$
    and $\mathbf z\in B$, we have
    \begin{equation}
    \begin{split}
        &|Q_t(\mathbf y,\mathbf z)-P(\mathbf z)|\\
        &\quad\lesssim
        \begin{cases}
            \displaystyle
            t^{-s_p-1}r^{s_p+1}
            w(B(\mathbf x_0,t))^{-1},
            &\displaystyle
            t>\frac{d(\mathbf x,\mathbf x_0)}{16},\\[6pt]
            \displaystyle
            t^{-s_p-1}r^{s_p+1}
            w(B(\mathbf x_0,d(\mathbf x,\mathbf x_0)))^{-1}
            e^{-c d(\mathbf x,\mathbf x_0)/t},
            &\displaystyle
            0<t\leq\frac{d(\mathbf x,\mathbf x_0)}{16}.
        \end{cases}
    \end{split}
    \end{equation}
    Consequently, by the cancellation condition of $a$,
    \begin{equation}
    \begin{split}
        |Q_t*a(\mathbf y)|
        &=
        \left|
            \int_B
            \bigl(Q_t(\mathbf y,\mathbf z)-P(\mathbf z)\bigr)
            a(\mathbf z)\,dw(\mathbf z)
        \right|\\
        &\lesssim
        \|a\|_{L^1(dw)}
        \begin{cases}
            \displaystyle
            t^{-s_p-1}r^{s_p+1}
            w(B(\mathbf x_0,t))^{-1},
            &\displaystyle
            t>\frac{d(\mathbf x,\mathbf x_0)}{16},\\[6pt]
            \displaystyle
            t^{-s_p-1}r^{s_p+1}
            w(B(\mathbf x_0,d(\mathbf x,\mathbf x_0)))^{-1}
            e^{-c d(\mathbf x,\mathbf x_0)/t},
            &\displaystyle
            0<t\leq\frac{d(\mathbf x,\mathbf x_0)}{16}.
        \end{cases}
    \end{split}
    \end{equation}
    Hence,
    \begin{equation}
    \begin{split}
        Sa(\mathbf x)^2
        &\lesssim
        r^{2(s_p+1)}\|a\|_{L^1(dw)}^2
        \int_0^{d(\mathbf x,\mathbf x_0)/16}
        \int_{B(\mathbf x,t)}
        \frac{t^{-2(s_p+1)}
              e^{-2c d(\mathbf x,\mathbf x_0)/t}}
             {w(B(\mathbf x_0,d(\mathbf x,\mathbf x_0)))^2}
        \frac{dw(\mathbf y)}{w(B(\mathbf x,t))}
        \frac{dt}{t}\\
        &\quad+
        r^{2(s_p+1)}\|a\|_{L^1(dw)}^2
        \int_{d(\mathbf x,\mathbf x_0)/16}^{\infty}
        \int_{B(\mathbf x,t)}
        \frac{t^{-2(s_p+1)}}{w(B(\mathbf x_0,t))^2}
        \frac{dw(\mathbf y)}{w(B(\mathbf x,t))}
        \frac{dt}{t}\\
        &\lesssim
        r^{2(s_p+1)}
        d(\mathbf x,\mathbf x_0)^{-2(s_p+1)}
        w(B(\mathbf x_0,d(\mathbf x,\mathbf x_0)))^{-2}
        \|a\|_{L^1(dw)}^2.
    \end{split}
    \end{equation}

    Finally, set
    \[
        U_j=
        \{\mathbf x\in\mathbb R^N:
        2^jr<d(\mathbf x,\mathbf x_0)\leq2^{j+1}r\},
        \qquad j\geq2.
    \]
    By $G$-invariance and doubling,
    \[
        w(U_j)\leq Cw(B(\mathbf x_0,2^jr)).
    \]
    Using also
    $\|a\|_{L^1(dw)}\leq w(B)^{1-1/p}$, we obtain
    \begin{equation*}
    \begin{split}
        &\int_{\{\mathbf x:\,d(\mathbf x,\mathbf x_0)>4r\}}
        |Sa(\mathbf x)|^p\,dw(\mathbf x)\\
        &\quad\lesssim
        \|a\|_{L^1(dw)}^p
        \sum_{j\geq2}\int_{U_j}
        r^{(s_p+1)p}
        d(\mathbf x,\mathbf x_0)^{-(s_p+1)p}
        w(B(\mathbf x_0,d(\mathbf x,\mathbf x_0)))^{-p}
        \,dw(\mathbf x)\\
        &\quad\lesssim
        \|a\|_{L^1(dw)}^p
        \sum_{j\geq2}
        r^{(s_p+1)p}(2^jr)^{-(s_p+1)p}
        w(B(\mathbf x_0,2^jr))^{1-p}\\
        &\quad\lesssim
        \sum_{j\geq2}
        2^{-j(s_p+1)p}
        w(B(\mathbf x_0,r))^{p-1}
        w(B(\mathbf x_0,2^jr))^{1-p}\\
        &\quad\lesssim
        \sum_{j\geq2}
        2^{-j(s_p+1)p}2^{j\mathbf N(1-p)}
        \leq C,
    \end{split}
    \end{equation*}
    because $(s_p+1)p>\mathbf N(1-p)$.
    Together with the local estimate, this completes
    the proof.
\end{proof}

\begin{theorem}\label{thm:atomic_inclusion}
Let $0<p\leq1$.
Under the natural identification with tempered distributions,
\[
    H^p_{\mathrm{CW},2}\subseteq H^p_{\mathrm{Dunkl}},
\]
and there exists a constant $C>0$ such that
\[
    \|f\|_{H^p_{\mathrm{Dunkl}}}
    \leq C\|f\|_{H^p_{\mathrm{CW},2}}
    \qquad(f\in H^p_{\mathrm{CW},2}).
\]
\end{theorem}

\begin{proof}
Let $f\in H^p_{\mathrm{CW},2}$ and choose a representation
\[
    f=\sum_{j=1}^{\infty}\lambda_ja_j
    \quad\text{in }\mathcal S'(\mathbb R^N),
    \qquad
    \sum_{j=1}^{\infty}|\lambda_j|^p<\infty.
\]
Set $f_n=\sum_{j=1}^n\lambda_ja_j$.
By the preceding proposition, each $a_j$ belongs to
$\mathbb H^p_{\mathrm{Dunkl}}$ and satisfies
$\|Sa_j\|_{L^p(dw)}\leq C$.
The sublinearity of $S$ and the inequality
$(\sum_j u_j)^p\leq\sum_j u_j^p$ for $u_j\geq0$ give,
for $n>m$,
\[
    \|f_n-f_m\|_{\mathbb H^p_{\mathrm{Dunkl}}}^p
    \leq
    \sum_{j=m+1}^{n}|\lambda_j|^p
    \|Sa_j\|_{L^p(dw)}^p
    \leq C\sum_{j=m+1}^{n}|\lambda_j|^p.
\]
Thus $(f_n)$ is Cauchy in $H^p_{\mathrm{Dunkl}}$
and converges to some $h\in H^p_{\mathrm{Dunkl}}$.
By the continuous embedding of $H^p_{\mathrm{Dunkl}}$
into $\mathcal S'(\mathbb R^N)$, this convergence also
holds in $\mathcal S'$.
Since the chosen representation gives $f_n\to f$
in $\mathcal S'$, uniqueness of the distributional
limit implies $h=f$.

Moreover,
\[
    \|f_n\|_{\mathbb H^p_{\mathrm{Dunkl}}}^p
    \leq C\sum_{j=1}^{n}|\lambda_j|^p.
\]
Passing to the limit and taking the infimum over all
atomic representations of $f$, we obtain
\[
    \|f\|_{H^p_{\mathrm{Dunkl}}}^p
    \leq C\|f\|_{H^p_{\mathrm{CW},2}}^p.
\]
This completes the proof.
\end{proof}

\section{Decomposition into \texorpdfstring{$({\rm CW},p,\infty)$}{(CW,p,infty)}-atoms}
\label{sec:bounded_atomic_decompositions}

Throughout this section, let $0<p\leq1$ and set
\[
    s=s_p=\left\lfloor\mathbf N\left(\frac1p-1\right)\right\rfloor.
\]
All polynomial spaces considered below are real.

\begin{proposition}\label{prop:atom_decomposition_linear}
There exists a constant $C>0$ such that every normalized
$({\rm CW},p,2)$-atom $a$ admits a representation
\[
    a=\sum_{k=1}^{\infty}\lambda_k b_k,
\]
where the $b_k$ are normalized $({\rm CW},p,\infty)$-atoms and
\[
    \sum_{k=1}^{\infty}|\lambda_k|^p\leq C.
\]
The series converges in $H^p_{\mathrm{CW},2}$ and
unconditionally in $\mathcal S'(\mathbb R^N)$.
The constant $C$ is independent of $a$.
\end{proposition}

\begin{proof}
By doubling, balls and cubes may be used interchangeably
in the definition of atoms, up to uniform normalization
constants. Indeed, every ball is contained in a cube
of comparable measure, and every cube is contained in
a ball of comparable measure. The support and cancellation
conditions are preserved under these enlargements.
We therefore first prove the assertion for atoms
associated with cubes and convert the resulting atoms
to ball atoms at the end.

It suffices to consider real-valued atoms. Indeed, if $a$
is complex-valued, then its nonzero real and imaginary
parts separately satisfy the same support, size, and
cancellation conditions. Applying the real-valued result
to these two functions gives the assertion for $a$,
with at most twice the bound for the sum of the
$p$th powers of the coefficients.

We shall use the following polynomial estimate:
\begin{equation}\label{eq:polynomial_cube_general}
    \|P\|_{L^\infty(Q)}
    \leq
    \frac{C_0}{w(Q)^{1/2}}\|P\|_{L^2(Q,dw)},
    \qquad P\in\mathcal P_s,
\end{equation}
uniformly over all cubes $Q$ (cf. Lemma~\ref{lem:polynomial_sup_L2})

Fix a real-valued cube atom $a$ associated with a cube $Q$.
Thus $a$ vanishes almost everywhere outside $Q$,
\[
    \|a\|_{L^2(dw)}
    \leq w(Q)^{1/2-1/p},
    \qquad
    \int_Q a(\mathbf x)P(\mathbf x)\,dw(\mathbf x)=0
    \quad(P\in\mathcal P_s).
\]
Choose $C_1\geq1$ such that
\[
    w(Q')\leq C_1w(Q'')
\]
whenever $Q''$ is a dyadic child of $Q'$.
Such a constant exists by the doubling property.

Let $0<\varepsilon<1$, to be fixed below, and put
\[
    \Lambda=\varepsilon^{-2}w(Q)^{-2/p}.
\]
Since
\[
    \frac1{w(Q)}\int_Q|a|^2\,dw
    \leq w(Q)^{-2/p}<\Lambda,
\]
the maximal dyadic subcubes $Q_j$ of $Q$ satisfying
\[
    \frac1{w(Q_j)}\int_{Q_j}|a|^2\,dw>\Lambda
\]
are proper and pairwise disjoint. Maximality gives
\begin{equation}\label{eq:CZ_1_linear}
    \Lambda<
    \frac1{w(Q_j)}\int_{Q_j}|a|^2\,dw
    \leq C_1\Lambda.
\end{equation}
Moreover, dyadic differentiation yields
\begin{equation}\label{eq:CZ_2_linear}
    |a|\leq\Lambda^{1/2}
    \quad\text{$dw$-almost everywhere on }
    \Omega=Q\setminus\bigcup_jQ_j,
\end{equation}
and
\begin{equation}\label{eq:CZ_bad_measure}
    \sum_jw(Q_j)
    \leq\Lambda^{-1}\|a\|_{L^2(dw)}^2
    \leq\varepsilon^2w(Q).
\end{equation}

For every $j$, let $P_{Q_j}a\in\mathcal P_s$ denote
the orthogonal projection of $a|_{Q_j}$ onto the
restrictions of $\mathcal P_s$ in the real Hilbert
space $L^2(Q_j,dw)$.
Then
\[
    \int_{Q_j}(a-P_{Q_j}a)P\,dw=0
    \qquad(P\in\mathcal P_s),
\]
and orthogonality gives
\[
    \|P_{Q_j}a\|_{L^2(Q_j,dw)}^2
    +
    \|a-P_{Q_j}a\|_{L^2(Q_j,dw)}^2
    =
    \|a\|_{L^2(Q_j,dw)}^2.
\]
By~\eqref{eq:polynomial_cube_general}
and~\eqref{eq:CZ_1_linear},
\begin{equation}\label{eq:proj_infinity_bound}
    \|P_{Q_j}a\|_{L^\infty(Q_j)}
    \leq
    C_0
    \left(\frac1{w(Q_j)}\int_{Q_j}|a|^2\,dw\right)^{1/2}
    \leq C_0C_1^{1/2}\Lambda^{1/2}.
\end{equation}

Set
\[
    g=a\chi_\Omega+\sum_j(P_{Q_j}a)\chi_{Q_j},
    \qquad
    h_j=(a-P_{Q_j}a)\chi_{Q_j}.
\]
The cubes $Q_j$ are disjoint, and the preceding
orthogonality identity implies
\[
    \sum_j\|h_j\|_{L^2(dw)}^2\leq\|a\|_{L^2(dw)}^2.
\]
Consequently,
\begin{equation}\label{eq:one_step_atom_decomposition}
    a=g+\sum_jh_j
    \quad\text{in }L^2(dw).
\end{equation}
Each $h_j$ has vanishing moments against $\mathcal P_s$.
The same is true of $g$, by
\eqref{eq:one_step_atom_decomposition} and the
cancellation of $a$; passage to the limit is justified
by testing against $P\chi_Q\in L^2(dw)$.

By~\eqref{eq:CZ_2_linear}
and~\eqref{eq:proj_infinity_bound},
\[
    \|g\|_\infty\leq C_2w(Q)^{-1/p},
    \qquad
    C_2=\varepsilon^{-1}
    \max\{1,C_0C_1^{1/2}\}.
\]
Thus $g=C_2b$, where $b$ is a normalized
$({\rm CW},p,\infty)$-atom associated with $Q$,
unless $g=0$, in which case this term is omitted.

For each nonzero $h_j$, define
\[
    \mu_j=\|h_j\|_{L^2(dw)}\,w(Q_j)^{1/p-1/2},
    \qquad
    a_j^*=\mu_j^{-1}h_j.
\]
Then $a_j^*$ is a normalized $({\rm CW},p,2)$-atom associated
with $Q_j$. Furthermore,
\[
    \|h_j\|_{L^2(dw)}
    \leq\|a\|_{L^2(Q_j,dw)}
    \leq C_1^{1/2}\Lambda^{1/2}w(Q_j)^{1/2},
\]
so
\[
    \mu_j
    \leq C_1^{1/2}\varepsilon^{-1}
    w(Q)^{-1/p}w(Q_j)^{1/p}.
\]
Using~\eqref{eq:CZ_bad_measure}, we obtain
\[
    \sum_j|\mu_j|^p
    \leq
    C_1^{p/2}\varepsilon^{-p}w(Q)^{-1}
    \sum_jw(Q_j)
    \leq C_1^{p/2}\varepsilon^{2-p}.
\]
Fix $\varepsilon$ sufficiently small that
\[
    C_1^{p/2}\varepsilon^{2-p}\leq\frac12.
\]
We have therefore proved the one-step decomposition
\begin{equation}\label{eq:recursive_atom_decomposition}
    a=C_2b+\sum_j\mu_ja_j^*,
    \qquad
    \sum_j|\mu_j|^p\leq\frac12.
\end{equation}

We iterate~\eqref{eq:recursive_atom_decomposition}
on the remaining $({\rm CW},p,2)$-atoms.
After $n$ steps, write
\[
    a=G_n+R_n,
    \qquad
    R_n=\sum_{\eta\in\mathcal I_n}c_\eta a_\eta,
\]
where $G_n$ consists of the bounded atoms obtained
during the first $n$ steps, and
\[
    \sum_{\eta\in\mathcal I_n}|c_\eta|^p\leq2^{-n}.
\]
The sum of the $p$th powers of the coefficients of the
bounded atoms produced at step $m+1$ is at most
$C_2^p2^{-m}$. Hence the full collection of bounded
atoms has coefficients satisfying
\begin{equation}\label{eq:bounded_atom_coefficients}
    \sum_k|\lambda_k|^p
    \leq C_2^p\sum_{m=0}^{\infty}2^{-m}
    =2C_2^p.
\end{equation}

For completeness, the remainder tends to zero
in $\mathcal S'(\mathbb R^N)$.
Indeed, the uniform test-function estimate in
Lemma~\ref{lem:distributional_convergence}, together
with the ball--cube comparison, gives
\[
    |\langle a_\eta,\varphi\rangle|
    \leq Cq_p(\varphi),
    \qquad
    q_p(\varphi)=\|\varphi\|_\infty+
    \sum_{|\beta|=s+1}\|\partial^\beta\varphi\|_\infty.
\]
Since $\ell^p\subseteq\ell^1$,
\[
    |\langle R_n,\varphi\rangle|
    \leq Cq_p(\varphi)
    \sum_{\eta\in\mathcal I_n}|c_\eta|
    \leq Cq_p(\varphi)2^{-n/p}\longrightarrow0.
\]
All series arising in the iteration are well-defined
in $\mathcal S'$ by the same estimate.
By~\eqref{eq:bounded_atom_coefficients}, the series of
bounded atoms converges unconditionally in $\mathcal S'$.
Its sum is $a$, since $a-G_n=R_n\to0$.

Finally, convert the cube atoms into ball atoms.
More explicitly, if $b$ is a bounded cube atom on $Q$
and $B_Q\supset Q$ is a ball with $w(B_Q)\leq Cw(Q)$,
then
\[
    \widetilde b=
    \left(\frac{w(Q)}{w(B_Q)}\right)^{1/p}b
\]
is a normalized $({\rm CW},p,\infty)$-atom on $B_Q$.
The resulting changes in the coefficients, as well
as the initial conversion from a ball atom to a cube
atom, are uniformly bounded.

Every normalized $({\rm CW},p,\infty)$-atom is also a normalized
$({\rm CW},p,2)$-atom. Therefore the resulting representation
satisfies
\[
    \Big\|a-\sum_{k=1}^{n}\lambda_kb_k
    \Big\|_{H^p_{\mathrm{CW},2}}^p
    \leq\sum_{k>n}|\lambda_k|^p\longrightarrow0.
\]
This proves the proposition.
\end{proof}

\begin{proof}[Proof of Theorem~\ref{teo:H_p_coincides}]
The previously proved inclusions give
\[
    H^p_{\mathrm{Dunkl}}=H^p_{\mathrm{CW},2}
\]
with equivalent quasi-norms. It remains to compare
the two atomic spaces.

Every normalized $({\rm CW},p,\infty)$-atom $b$ associated with
a ball $B$ satisfies
\[
    \|b\|_{L^2(dw)}
    \leq w(B)^{1/2}\|b\|_\infty
    \leq w(B)^{1/2-1/p}.
\]
Thus it is also a normalized $({\rm CW},p,2)$-atom, and
\[
    \|f\|_{H^p_{\mathrm{CW},2}}
    \leq\|f\|_{H^p_{\mathrm{CW},\infty}}.
\]

Conversely, let $f\in H^p_{\mathrm{CW},2}$ and take
an atomic representation
\[
    f=\sum_j\beta_ja_j
    \quad\text{in }\mathcal S'(\mathbb R^N),
    \qquad
    \sum_j|\beta_j|^p<\infty.
\]
By Proposition~\ref{prop:atom_decomposition_linear},
\[
    a_j=\sum_k\gamma_{j,k}b_{j,k},
    \qquad
    \sum_k|\gamma_{j,k}|^p\leq C,
\]
where the $b_{j,k}$ are normalized $({\rm CW},p,\infty)$-atoms.
Consequently,
\[
    \sum_{j,k}|\beta_j\gamma_{j,k}|^p
    \leq C\sum_j|\beta_j|^p.
\]
Lemma~\ref{lem:distributional_convergence} shows that
the corresponding double series converges
unconditionally in $\mathcal S'$.
For every $\varphi\in\mathcal S(\mathbb R^N)$,
absolute convergence of the scalar series permits
grouping the terms:
\[
    \sum_{j,k}\beta_j\gamma_{j,k}
        \langle b_{j,k},\varphi\rangle
    =
    \sum_j\beta_j\langle a_j,\varphi\rangle
    =
    \langle f,\varphi\rangle.
\]
Hence this double series is a bounded atomic
representation of $f$, and
\[
    \|f\|_{H^p_{\mathrm{CW},\infty}}^p
    \leq C\sum_j|\beta_j|^p.
\]
Taking the infimum over the original representations
of $f$ gives
\[
    \|f\|_{H^p_{\mathrm{CW},\infty}}
    \leq C\|f\|_{H^p_{\mathrm{CW},2}}.
\]
Together with the preceding inclusions, this proves
\[
    H^p_{\mathrm{Dunkl}}
    =
    H^p_{\mathrm{CW},2}
    =
    H^p_{\mathrm{CW},\infty}
\]
with equivalent quasi-norms.
\end{proof}

\section*{Declarations}

\subsection*{Use of Generative AI and AI-Assisted Technologies}
During the preparation of this manuscript, the authors utilized Google's Gemini (Thinking 3.6) and OpenAI's GPT (GPT-6-Astra) to assist with initial proof strategy brainstorming, language refinement, typesetting equations, and drafting preliminary standard technical lemmas. 

The authors reviewed and edited all content as necessary and takes full responsibility for the accuracy, mathematical correctness, and overall integrity of the final published work.

\subsection*{Competing Interests}
The authors declare no competing financial or non-financial interests directly relevant to the content of this article.

\subsection*{Data Availability}
Data sharing is not applicable to this article as no datasets were generated or analyzed during the current study.

\end{document}